\documentclass[12pt,a4paper]{article}

\usepackage{amssymb,amsxtra,amsthm,amsmath}

\usepackage[a4paper,margin=1.5cm]{geometry}

\usepackage[labelsep=none]{caption}

\swapnumbers

\theoremstyle{theorem}
\newtheorem{theorem}[subsection]{Theorem}
\newtheorem{corollary}[subsection]{Corollary}
\newtheorem{lemma}[subsection]{Lemma}
\newtheorem{proposition}[subsection]{Proposition}

\theoremstyle{definition}
\newtheorem{definition}[subsection]{Definition}
\newtheorem{example}[subsection]{Example}
\newtheorem{remark}[subsection]{Remark}
\newtheorem*{acknowledgement}{Acknowledgements}
\newtheorem*{notation}{Notation}

\numberwithin{figure}{section}
\numberwithin{equation}{section}

\usepackage{ifpdf}
\ifpdf
\usepackage[pdftex]{graphics}
\IfFileExists{hyperref.sty}{\usepackage[pdftex]{hyperref}}{}
\else
\usepackage[dvips]{graphics}
\IfFileExists{hyperref.sty}{\usepackage[hypertex]{hyperref}}{}
\fi

\usepackage[mathcal]{euscript}
\newcommand{\ca}{{\mathcal A}}
\newcommand{\cb}{{\mathcal B}}
\newcommand{\cc}{{\mathcal C}}
\newcommand{\cd}{{\mathcal D}}
\newcommand{\ce}{{\mathcal E}}

\newcommand{\cv}{{\mathcal V}}
\newcommand{\cw}{{\mathcal W}}

\usepackage[sans]{dsfont}
\newcommand{\unito}{{\mathds 1}}

\makeatletter
\renewcommand\subsection{\@startsection{subsection}{2}{\z@}%
  {\topsep}%
  {-1.5ex \@plus -.2ex}%
  {\normalfont\normalsize\bfseries}}
\renewcommand\paragraph{\@startsection{paragraph}{4}{\z@}%
  {\topsep}%
  {-1em}%
  {\normalfont\normalsize\bfseries}}
\makeatother

\let\rto\xrightarrow
\let\tens\otimes
\let\und\underline
\let\wh\widehat

\newcommand\NN{{\mathbb N}}
\newcommand{\Chi}{{\mathrm{X}}}
\newcommand{\n}[1]{\nobreakdash-\hspace{0pt}}
\newcommand{\ainf}[1]{$A_\infty$\nobreakdash-\hspace{0pt}}

\newcommand\mcC{{\mathsf C}}
\newcommand\mcD{{\mathsf D}}
\newcommand\mcE{{\mathsf E}}

\newcommand\mcV{{\mathsf V}}
\newcommand\mcW{{\mathsf W}}

\newcommand{\Cat}{{\mathbf{Cat}}}
\newcommand{\Set}{{\mathcal{S}}}
\newcommand{\Multicat}{{\mathbf{Multicat}}}

\newcommand{\ClCat}{{\mathbf{ClCat}}}
\newcommand{\ClMulticat}{{\mathbf{ClMulticat}}}
\newcommand{\VCat}{{\cv\text-\Cat}}

\DeclareMathOperator\ev{ev}
\DeclareMathOperator\Hom{Hom}
\DeclareMathOperator\id{id}
\DeclareMathOperator\Id{Id}
\DeclareMathOperator\Ob{Ob}
\newcommand{\op}{{\operatorname{op}}}

\newcommand{\defeq}{\overset{\textup{def}}{=}}

\newcommand{\corref}[1]{Corollary~\ref{#1}}
\newcommand{\defref}[1]{Definition~\ref{#1}}
\newcommand{\diaref}[1]{Diagram~\ref{#1}}
\newcommand{\exaref}[1]{Example~\ref{#1}}

\newcommand{\lemref}[1]{Lemma~\ref{#1}}
\newcommand{\propref}[1]{Proposition~\ref{#1}}
\newcommand{\remref}[1]{Remark~\ref{#1}}

\newcommand{\thmref}[1]{Theorem~\ref{#1}}

\usepackage[all,dvips]{xy}
\SelectTips{cm}{10}

\begin{document}
\bibliographystyle{amsplain}

\title{Closed categories vs. closed multicategories} 
\author{Oleksandr Manzyuk} 

\maketitle

\begin{abstract}
  We prove that the 2\n-category of closed categories of Eilenberg and
  Kelly is equivalent to a suitable full 2\n-subcategory of the
  2\n-category of closed multicategories.
\end{abstract}

\section{Introduction}

The notion of closed category was introduced by Eilenberg and Kelly
\cite{EK}. It is an axiomatization of the notion of category with
internal function spaces. More precisely, a closed category is a
category \(\cc\) equipped with a functor
\(\und\cc(-,-):\cc^\op\times\cc\to\cc\), called the \emph{internal
  \(\Hom\)\n-functor}; an object \(\unito\) of \(\cc\), called the
\emph{unit object}; a natural isomorphism
\(i_X:X\rto\sim\und\cc(\unito,X)\), and natural transformations
\(j_X:\unito\to\und\cc(X,X)\) and
\(L^X_{YZ}:\und\cc(Y,Z)\to\und\cc(\und\cc(X,Y),\und\cc(X, Z))\). These
data are to satisfy five axioms; see \defref{def:closed-category} for
details.

A wide class of examples is provided by closed monoidal categories. We
recall that a monoidal category \(\cc\) is called \emph{closed} if for
each object \(X\) of \(\cc\) the functor \(X\tens-\) admits a right
adjoint \(\und\cc(X,-)\); i.e, there exists a bijection \(\cc(X\tens
Y,Z)\cong\cc(Y,\und\cc(X,Z))\) that is natural in both \(Y\) and
\(Z\). Equivalently, a monoidal category \(\cc\) is closed if and only
if for each pair of objects \(X\) and \(Z\) of \(\cc\) there exist an
\emph{internal \(\Hom\)\n-object} \(\und\cc(X,Z)\) and an
\emph{evaluation} morphism \(\ev^\cc_{X,Z}:X\tens\und\cc(X,Z)\to Z\)
satisfying the following universal property: for each morphism
\(f:X\tens Y\to Z\) there exists a unique morphism
\(g:Y\to\und\cc(X,Z)\) such that \(f=\ev^\cc_{X,Z}\circ(1_X\tens
g)\). One can check that the map \((X,Z)\mapsto\und\cc(X,Z)\) extends
uniquely to a functor \(\und\cc(-,-):\cc^\op\times\cc\to\cc\), which
together with certain canonically chosen transformations \(i_X\),
\(j_X\), and \(L^X_{YZ}\) turns \(\cc\) into a closed category.

While closed monoidal categories are in prevalent use in mathematics,
arising in category theory, algebra, topology, analysis, logic, and
theoretical computer science, there are also important examples of
closed categories that are not monoidal. The author's motivation
stemmed from the theory of \ainf-categories.

The notion of \ainf-category appeared at the beginning of the nineties
in the work of Fukaya on Floer homology~\cite{Fukaya:A-infty}. However
its precursor, the notion of \ainf-algebra, was introduced in the
early sixties by Stasheff~\cite{Stasheff:HomAssoc}. It as a
linearization of the notion of \ainf-space, a topological space
equipped with a product operation which is associative up to homotopy,
and the homotopy which makes the product associative can be chosen so
that it satisfies a collection of higher coherence conditions. Loosely
speaking, \ainf-categories are to \ainf-algebras what linear
categories are to algebras. On the other hand, \ainf-categories
generalize differential graded categories. Unlike in differential
graded categories, in \ainf-categories composition need not be
associative on the nose; it is only required to be associative up to
homotopy that satisfies a certain equation up to another homotopy, and
so on.

Many properties of \ainf-categories follow from the discovery,
attributed to Kontsevich, that for each pair of \ainf-categories
\(\ca\) and \(\cb\) there is a natural \ainf-category
\(A_\infty(\ca,\cb)\) with \ainf-functors from \(\ca\) to \(\cb\) as
its objects.  These \ainf-categories of \ainf-functors were also
investigated by many other authors,
e.g. Fukaya~\cite{Fukaya:FloerMirror-II},
Lef\`evre-Hasegawa~\cite{Lefevre-Ainfty-these}, and
Lyubashenko~\cite{Lyu-AinfCat}; they allow us to equip the category of
\ainf-categories with the structure of a closed category.

In the recent monograph by Bespalov, Lyubashenko, and the author
\cite{BLM} the theory of \ainf-categories is developed from a slightly
different perspective. Our approach is based on the observation that
although the category of \ainf-categories is not monoidal, there is a
natural notion of \ainf-functor of many arguments, and thus
\ainf-categories form a \emph{multicategory}. The notion of
multicategory (known also as colored operad or pseudo-tensor category)
was introduced by Lambek \cite{Lambek:DedSys,Lambek:Multi}. It is a
many\n-object version of the notion of operad. If morphisms in a
category are considered as analogous to functions, morphisms in a
multicategory are analogous to functions in several variables. The
most familiar example of multicategory is the multicategory of vector
spaces and multilinear maps. An arrow in a multicategory looks like
\(X_1,X_2,\dots,X_n\to Y\), with a finite sequence of objects as the
domain and one object as the codomain.  Multicategories generalize
monoidal categories: monoidal category \(\mathcal C\) gives rise to a
multicategory \(\widehat{\mathcal C}\) whose objects are those of
\(\mathcal C\) and whose morphisms \(X_1,X_2,\dots,X_n\to Y\) are
morphisms \(X_1\otimes X_2\otimes\dots\otimes X_n\to Y\) of \(\mathcal
C\). The notion of closedness for multicategories is a straightforward
generalization of that for monoidal categories. We say that a
multicategory \(\mcC\) is \emph{closed} if for each sequence
\(X_1,\dots,X_m,Z\) of objects of \(\mcC\) there exist an internal
\(\Hom\)\n-object \(\und\mcC(X_1,\dots,X_m;Z)\) and an evaluation
morphism \( \ev^\mcC_{X_1,\dots,X_m;Z}: X_1,\dots,X_m,
\und\mcC(X_1,\dots,X_m;Z)\to Z \) satisfying the following universal
property: for each morphism \(f:X_1,\dots,X_m,Y_1,\dots,Y_n\to Z\)
there is a unique morphism
\(g:Y_1,\dots,Y_n\to\und\mcC(X_1,\dots,X_m;Z)\) such that
\(f=\ev^\mcC_{X_1,\dots,X_m;Z}\circ(1_{X_1},\dots,1_{X_m},g)\).  We
prove that the multicategory of \ainf-categories is closed, thus
obtaining a conceptual explanation of the origin of the
\ainf-categories of \ainf-functors.

The definition of closed multicategory seems to be sort of a
mathematical folklore, and to the best of the author's knowledge it
did not appear in press before \cite{BLM}. The only reference the
author is aware of is the paper of Hyland and Power on pseudo-closed
2\n-categories \cite{HylandPower}, where the notion of closed
\(\Cat\)\n-multicategory (i.e., multicategory enriched in the category
\(\Cat\) of categories) is implicitly present, although not spelled
out.

This paper arose as an attempt to understand in general the relation
between closed categories and closed multicategories. It turned out
that these notions are essentially equivalent in a very strong
sense. Namely, on the one hand, there is a 2\n-category of closed
categories, closed functors, and closed natural transformations. On
the other hand, there is a 2\n-category of closed
multi\-ca\-te\-go\-ries with unit objects, multifunctors, and
multinatural transformations. Because a 2\n-category is the same thing
as a category enriched in \(\Cat\), it makes sense to speak about
\(\Cat\)\n-functors between 2\n-categories; these can be called strict
2\n-functors because they preserve composition of 1\n-morphisms and
identity 1\n-morphisms strictly. We construct a \(\Cat\)\n-functor
from the 2\n-category of closed multicategories with unit objects to
the 2\n-category of closed categories, and prove that it is a
\(\Cat\)\n-equivalence; see \propref{prop-Cat-fun-cl-multi-cl-cat} and
\thmref{thm-equiv}.

Both closed categories and multicategories can bear symmetries. With
some additional work it can be proven that the 2\n-category of
symmetric closed categories is \(\Cat\)\n-equivalent to the
2\n-category of symmetric closed multicategories with unit objects. We
are not going to explore this subject here.

We should mention that the definition of closed category we adopt in
this paper does not quite agree with the definition appearing in
\cite{EK}. Closed categories have been generalized by Street
\cite{Street:cosmoi} to extension systems; a closed category in our
sense is an extension system with precisely one object.  We discuss
carefully the relation between these definitions because it is crucial
for our proof of \thmref{thm-equiv}; see \remref{rem-cl-cat-EK} and
\propref{thm-cl-cat-iso-EK}. Our definition of closed category also
coincides with the definition appearing in Laplaza's paper
\cite{Laplaza}, to which we would like to pay special tribute because
it allowed us to give an elegant construction of a closed
multicategory with a given underlying closed category.

\begin{notation}
  We use interchangeably the notations \(g\circ f\) and \(f\cdot g\)
  for the composition of morphisms \(f:X\to Y\) and \(g:Y\to Z\) in a
  category, giving preference to the latter notation, which is more
  readable. Throughout the paper the set of nonnegative integers is
  denoted by \(\NN\), the category of sets is denoted by \(\Set\), and
  the category of categories is denoted by \(\Cat\).
\end{notation}

\begin{acknowledgement}
  I would like to thank Volodymyr Lyubashenko and Yuri Bespalov for
  many fruitful discussions. This work was written up during my stay
  at York University. I would like to thank Professor Walter Tholen
  for inviting me to York and for carefully reading preliminary
  versions of this paper.
\end{acknowledgement}

\section{Closed categories}
\label{sec:closed-categories}

In this section we give preliminaries on closed categories.  We begin
by recalling the definition of closed category appearing in
\cite[Section~4]{Street:cosmoi} and \cite{Laplaza}.

\begin{definition}
  \label{def:closed-category}
  A \emph{closed category} \((\cc,\und\cc(-,-),\unito,i,j,L)\) consists of the following data:
  \begin{itemize}
  \item a category \(\cc\);
  \item a functor \(\und\cc(-,-):\cc^\op\times\cc\to\cc\);
  \item an object \(\unito\) of \(\cc\);
  \item a natural isomorphism
    \(i:\Id_\cc\rto\sim\und\cc(\unito,-):\cc\to\cc\);
  \item a transformation \(j_X:\unito\to\und\cc(X,X)\), dinatural in
    \(X\in\Ob\cc\);
  \item a transformation
    \(L^X_{YZ}:\und\cc(Y,Z)\to\und\cc(\und\cc(X,Y),\und\cc(X,Z))\),
    natural in \(Y,Z\in\Ob\cc\)  and dinatural in \(X\in\Ob\cc\).
  \end{itemize}
  These data are subject to the following axioms.
  \begin{list}{}{%
      \setlength{\labelwidth}{3.5em}
      \setlength{\leftmargin}{4em}}
  \item[CC1.] The following equation holds true:
    \begin{equation*}
      \bigl[
      \unito\rto{j_Y}\und\cc(Y,Y)\rto{L^X_{YY}}\und\cc(\und\cc(X,Y),\und\cc(X,Y))
      \bigr]=j_{\und\cc(X,Y)}.
    \end{equation*}

  \item[CC2.] The following equation holds true:
    \begin{equation*}
      \bigl[
      \und\cc(X,Y)\rto{L^X_{XY}}\und\cc(\und\cc(X,X),\und\cc(X,Y))\rto{\und\cc(j_X,1)}\und\cc(\unito,\und\cc(X,Y))
      \bigr]=i_{\und\cc(X,Y)}.
    \end{equation*}

  \item[CC3.] The following diagram commutes:
    \begin{equation*}
      \begin{xy}
        (-45,18)*+{\und\cc(U,V)}="1";
        (45,18)*+{\und\cc(\und\cc(Y,U),\und\cc(Y,V))}="2";
        (-45,0)*+{\und\cc(\und\cc(X,U),\und\cc(X,V))}="3";
        (-45,-18)*+{\und\cc(\und\cc(\und\cc(X,Y),\und\cc(X,U)),\und\cc(\und\cc(X,Y),\und\cc(X,V)))}="4";
        (45,-18)*+{\und\cc(\und\cc(Y,U),\und\cc(\und\cc(X,Y),\und\cc(X,V)))}="5";
        {\ar@{->}^-{L^Y_{UV}} "1";"2"};
        {\ar@{->}_-{L^X_{UV}} "1";"3"};
        {\ar@{->}^-{\und\cc(1,L^X_{YV})} "2";"5"};
        {\ar@{->}_-{L^{\und\cc(X,Y)}_{\und\cc(X,U),\und\cc(X,V)}} "3";"4"};
        {\ar@{->}^-{\und\cc(L^X_{YU},1)} "4";"5"};
      \end{xy}
    \end{equation*}

  \item[CC4.] The following equation holds true:
    \begin{equation*}
      \bigl[
      \und\cc(Y,Z)\rto{L^\unito_{YZ}}\und\cc(\und\cc(\unito,Y),\und\cc(\unito,Z))
      \rto{\und\cc(i_Y,1)}\und\cc(Y,\und\cc(\unito,Z))
      \bigr]=\und\cc(1,i_Z).
    \end{equation*}

  \item[CC5.] The map \(\gamma:\cc(X,Y)\to\cc(\unito,\und\cc(X,Y))\)
    that sends a morphism \(f:X\to Y\) to the composite
    \begin{equation*}
      \unito\rto{j_X}\und\cc(X,X)\rto{\und\cc(1,f)}\und\cc(X,Y)
    \end{equation*}
    is a bijection.
  \end{list}
  We shall call \(\und\cc(-,-)\) the \emph{ internal
    \(\Hom\)\n-functor} and \(\unito\) the \emph{unit object}.
\end{definition}

\begin{example}\label{exa-set-closed-cat}
  The category \(\Set\) of sets becomes a closed category if we set
  \(\und\Set(-,-)=\Set(-,-)\); take for \(\unito\) a set \(\{*\}\),
  chosen once and for all, consisting of a single point \(*\); and
  define \(i\), \(j\), \(L\) by:
  \begin{alignat*}{3}
    i_X(x)(*) & =x, & \qquad & x\in X;\\
    j_X(*) & =1_X; && \\
    L^X_{YZ}(g)(f) & = f\cdot g, & \qquad & f\in\Set(X,Y),\quad
    g\in\Set(Y,Z).
  \end{alignat*}
\end{example}

\begin{remark}\label{rem-cl-cat-EK}
  Definition~\ref{def:closed-category} is slightly different from the
  original definition by Eilenberg and Kelly
  \cite[Section~2]{EK}. They require that a closed category \(\cc\) be
  equipped with a functor \(C:\cc\to\Set\) such that the following
  axioms are satisfied in addition to CC1--CC4.
  \begin{list}{}{%
      \setlength{\labelwidth}{3.5em}
      \setlength{\leftmargin}{4em}}
  \item[CC0.] The following diagram of functors commutes:
    \[
    \begin{xy}
      (-17,9)*+{\cc^\op\times\cc}="1"; 
      (17,9)*+{\cc}="2";
      (17,-9)*+{\Set}="3"; 
      {\ar@{->}^-{\und\cc(-,-)} "1";"2"};
      {\ar@{->}_-{\cc(-,-)} "1";"3"}; 
      {\ar@{->}^-{C} "2";"3"};
    \end{xy}
    \]

  \item[CC5'.] The map
    \[
    Ci_{\und\cc(X,X)}:\cc(X,X)=C\und\cc(X,X)\to C\und\cc(\unito,\und\cc(X,X))=\cc(\unito,\und\cc(X,X))
    \]
    sends \(1_X\in\cc(X,X)\) to \(j_X\in\cc(\unito,\und\cc(X,X))\).
  \end{list}
  \cite[Lemma~2.2]{EK} implies that
  \[
  \gamma=Ci_{\und\cc(X,Y)}:\cc(X,Y)=C\und\cc(X,Y)\to
  C\und\cc(\unito,\und\cc(X,Y))=\cc(\unito,\und\cc(X,Y)),
  \]
  so that a closed category in the sense of Eilenberg and Kelly is
  also a closed category in our sense. Furthermore, as we shall see
  later, an arbitrary closed category in our sense is isomorphic to a
  closed category in the sense of Eilenberg and Kelly.
\end{remark}

\begin{proposition}[{\cite[Proposition~2.5]{EK}}]\label{prop-iC1X-C1iX}
  \(i_{\und\cc(\unito,X)}=\und\cc(1,i_X):
  \und\cc(\unito,X)\to\und\cc(\unito,\und\cc(\unito,X))\).
\end{proposition}

\begin{proof}
  The proof given in \cite[Proposition~2.5]{EK} translates word by word to our setting.
\end{proof}

\begin{proposition}[{\cite[Proposition~2.7]{EK}}]\label{prop-j1-i1}
  \(j_\unito=i_\unito:\unito\to\und\cc(\unito,\unito)\).
\end{proposition}

\begin{proof}
  The proof given in \cite[Proposition~2.7]{EK} relies on the axiom
  CC5', and thus is not applicable here; we give an independent proof.
  The map \( \gamma:\cc(\unito,\und\cc(\unito,\unito))\to
  \cc(\unito,\und\cc(\unito,\und\cc(\unito,\unito))) \) is a bijection
  by the axiom CC5, so it suffices to prove that
  \(\gamma(j_\unito)=\gamma(i_\unito)\). We have:
  \begin{alignat*}{3}
    \gamma(i_\unito)&=\bigl[\unito\rto{j_\unito}\und\cc(\unito,\unito)\rto{\und\cc(1,i_\unito)}
    \und\cc(\unito,\und\cc(\unito,\unito))\bigr] & &
    \\
    &=\bigl[\unito\rto{j_\unito}\und\cc(\unito,\unito)\rto{i_{\und\cc(\unito,\unito)}}
    \und\cc(\unito,\und\cc(\unito,\unito))\bigr] &\quad &
    \textup{(\propref{prop-iC1X-C1iX})}
    \\
    &=\bigl[
    \unito\rto{j_\unito}\und\cc(\unito,\unito)\rto{L^\unito_{\unito\unito}}
    \und\cc(\und\cc(\unito,\unito),\und\cc(\unito,\unito))
    \rto{\und\cc(j_\unito,1)}\und\cc(\unito,\und\cc(\unito,\unito))
    \bigr] & \quad & \textup{(axiom CC2)}
    \\
    &=\bigl[
    \unito\rto{j_{\und\cc(\unito,\unito)}}\und\cc(\und\cc(\unito,\unito),\und\cc(\unito,\unito))
    \rto{\und\cc(j_\unito,1)}\und\cc(\unito,\und\cc(\unito,\unito))
    \bigr] & \quad & \textup{(axiom CC1)}
    \\
    &=\bigl[
    \unito\rto{j_\unito}\und\cc(\unito,\unito)\rto{\und\cc(1,j_\unito)}\und\cc(\unito,\und\cc(\unito,\unito))
    \bigr] & \quad & \textup{(dinaturality of \(j\))}
    \\
    &=\gamma(j_\unito).
  \end{alignat*}
  The proposition is proven.
\end{proof}

\begin{corollary}\label{cor-gamma-C-1i-1}
  \(\bigl[ \cc(\unito,X) \rto{\gamma} \cc(\unito,\und\cc(\unito,X))
  \rto{\cc(\unito,i^{-1}_X)} \cc(\unito,X) \bigr]=1_{\cc(\unito,X)}\).
\end{corollary}

\begin{proof}
  An element \(f\in\cc(\unito,X)\) is mapped by the left hand side to the
  composite
  \[
  \unito \rto{j_\unito} \und\cc(\unito,\unito) \rto{\und\cc(1,f)}
  \und\cc(\unito,X) \rto{i^{-1}_X}X,
  \]
  which is equal to
  \[
  \bigl[ \unito \rto{j_\unito} \und\cc(\unito,\unito)
  \rto{i^{-1}_\unito} \unito \rto{f}X \bigr]=f
  \]
  by the naturality of \(i^{-1}_X\), and because
  \(j_\unito=i_\unito:\unito\to\und\cc(\unito,\unito)\) by
  \propref{prop-j1-i1}. The corollary is proven.
\end{proof}

\begin{proposition}\label{prop-gamma-C1L-undC-gamma}
  The following diagram commutes:
  \[
  \begin{xy}
    (-27,9)*+{\cc(Y,Z)}="1";
    (27,9)*+{\cc(\und\cc(X,Y),\und\cc(X,Z))}="2";
    (-27,-9)*+{\cc(\unito,\und\cc(Y,Z))}="3";
    (27,-9)*+{\cc(\unito,\und\cc(\und\cc(X,Y),\und\cc(X,Z)))}="4";
    {\ar@{->}^-{\und\cc(X,-)} "1";"2"};
    {\ar@{->}_-{\gamma} "1";"3"};
    {\ar@{->}^-{\gamma} "2";"4"};
    {\ar@{->}^-{\cc(\unito,L^X_{YZ})} "3";"4"};
  \end{xy}
  \]
\end{proposition}

\begin{proof}
  For each \(f\in\cc(Y,Z)\),  we have:
  \begin{align*}
    \cc(\unito,L^X_{YZ})(\gamma(f))&=
    \bigl[ \unito \rto{j_Y} \und\cc(Y,Y) \rto{\und\cc(1,f)} \und\cc(Y,Z)
    \rto{L^X_{YZ}} \und\cc(\und\cc(X,Y),\und\cc(X,Z)) \bigr]
    \\
    &=\bigl[
    \unito\rto{j_Y}\und\cc(Y,Y)\rto{L^X_{YY}}\und\cc(\und\cc(X,Y),\und\cc(X,Y))\rto{\und\cc(1,\und\cc(1,f))}
    \und\cc(\und\cc(X,Y),\und\cc(X,Z))
    \bigr]
    \\
    &=\bigl[
    \unito\rto{j_{\und\cc(X,Y)}}\und\cc(\und\cc(X,Y),\und\cc(X,Y))\rto{\und\cc(1,\und\cc(1,f))}\und\cc(\und\cc(X,Y),\und\cc(X,Z))
    \bigr]
    \\
    &=\gamma(\und\cc(1,f)),
  \end{align*}
  where the second equality is by the dinaturality of \(L^X_{YZ}\) in
  \(X\), and the third equality is by the axiom~CC1.
\end{proof}

\begin{proposition}\label{prop-gamma-fg-gamma-f-C1g}
  For each \(f\in\cc(X,Y)\), \(g\in\cc(Y,Z)\), we have \(\gamma(f\cdot
  g) = \gamma(f)\cdot\und\cc(1,g)=\gamma(g)\cdot\und\cc(f,1)\).
\end{proposition}

\begin{proof}
  Indeed, \(\gamma(f\cdot g) = j_X\cdot\und\cc(1,f\cdot g) =
  j_X\cdot\und\cc(1,f)\cdot\und\cc(1,g) =
  \gamma(f)\cdot\und\cc(1,g)\), proving the first equality. Let us
  prove the second equality. We have:
  \begin{alignat*}{3}
    \gamma(f)\cdot\und\cc(1,g) & = \bigl[ \unito \rto{j_X}
    \und\cc(X,X) \rto{\und\cc(1,f)} \und\cc(X,Y) \rto{\und\cc(1,g)}
    \und\cc(X,Z) \bigr] & &
    \\
    & = \bigl[ \unito \rto{j_Y} \und\cc(Y,Y) \rto{\und\cc(f,1)}
    \und\cc(X,Y) \rto{\und\cc(1,g)} \und\cc(X,Z) \bigr] & \quad &
    \textup{(dinaturality of \(j\))}
    \\
    & = \bigl[ \unito \rto{j_Y} \und\cc(Y,Y) \rto{\und\cc(1,g)}
    \und\cc(Y,Z) \rto{\und\cc(f,1)} \und\cc(X,Z) \bigr] & \quad &
    \textup{(functoriality of \(\und\cc(-,-)\))}
    \\
    & = \gamma(g)\cdot\und\cc(f,1). &&
  \end{alignat*}
  The proposition is proven.
\end{proof}

We now recall the definitions of closed functor and closed natural
transformation following \cite[Section~2]{EK}. 

\begin{definition}
  Let \(\cc\) and \(\cd\) be closed categories. A \emph{closed functor}
  \(\Phi=(\phi,\hat\phi,\phi^0):\cc\to\cd\) consists of the following
  data:
  \begin{itemize}
  \item a functor \(\phi:\cc\to\cd\);
  \item a natural transformation
    \(\hat\phi=\hat\phi_{X,Y}:\phi\und\cc(X,Y)\to\und\cd(\phi X,\phi Y)\);
  \item a morphism \(\phi^0:\unito\to\phi\unito\).
  \end{itemize}
  These data are subject to the following axioms.
  \begin{list}{}{%
      \setlength{\labelwidth}{3.5em}
      \setlength{\leftmargin}{4em}}
    
  \item[CF1.] The following equation holds true:
    \begin{equation*}
      \bigl[
      \unito\rto{\phi^0}\phi\unito\rto{\phi j_X}\phi\und\cc(X,X)\rto{\hat\phi}\und\cd(\phi X,\phi X)
      \bigr]=j_{\phi X}.
    \end{equation*}

  \item[CF2.] The following equation holds true:
    \begin{equation*}
      \bigl[ \phi X\rto{\phi
        i_X}\phi\und\cc(\unito,X)\rto{\hat\phi}\und\cd(\phi\unito,\phi
      X)\rto{\und\cd(\phi^0,1)}\und\cd(\unito,\phi X)
      \bigr]=i_{\phi X}.
    \end{equation*}

  \item[CF3.] The following diagram commutes:
    \begin{equation*}
      \begin{xy}
        (-50,9)*+{\phi\und\cc(Y,Z)}="1";
        (-0,9)*+{\phi\und\cc(\und\cc(X,Y),\und\cc(X,Z))}="2";
        (65,9)*+{\und\cd(\phi\und\cc(X,Y),\phi\und\cc(X,Z))}="3";
        (-50,-9)*+{\und\cd(\phi Y,\phi Z)}="4";
        (-0,-9)*+{\und\cd(\und\cd(\phi X,\phi Y),\und\cd(\phi X,\phi Z))}="5";
        (65,-9)*+{\und\cd(\phi\und\cc(X,Y),\und\cd(\phi X,\phi Z))}="6";
        {\ar@{->}^-{\phi L^X_{YZ}} "1";"2"};
        {\ar@{->}^-{\hat\phi} "2";"3"};
        {\ar@{->}_-{\hat\phi} "1";"4"};
        {\ar@{->}^-{L^{\phi X}_{\phi Y,\phi Z}} "4";"5"};
        {\ar@{->}^-{\und\cd(\hat\phi,1)} "5";"6"};
        {\ar@{->}^-{\und\cd(1,\hat\phi)} "3";"6"};
      \end{xy}
    \end{equation*}
  \end{list}
\end{definition}

\begin{proposition}\label{prop-E-V1?}
  Let \(\cv\) be a closed category. There is a closed functor
  \(E=(e,\hat e,e^0):\cv\to\Set\), where:
  \begin{itemize}
  \item \(e=\cv(\unito,-):\cv\to\Set\);
  \item \(\hat e=\bigl[\cv(\unito,\und\cv(X,Y)) \rto{\gamma^{-1}}
    \cv(X,Y) \rto{\cv(\unito,-)}
    \Set(\cv(\unito,X),\cv(\unito,X))\bigr]\);
  \item \(e^0:\{*\}\to\cv(\unito,\unito)\), \(*\mapsto 1_\unito\).
  \end{itemize}
\end{proposition}

\begin{proof}
  Let us check the axioms CF1--CF3. The reader is referred to
  \exaref{exa-set-closed-cat} for a description of the structure of
  a closed category on \(\Set\).

  \paragraph{CF1.} We must prove the following equation:
  \[
  \bigl[
  \{*\}\rto{e^0}\cv(\unito,\unito)\rto{\cv(\unito,j_X)}\cv(\unito,\und\cv(X,X))\rto{\gamma^{-1}}
  \cv(X,X)\rto{\cv(\unito,-)}\Set(\cv(\unito,X),\cv(\unito,X))
  \bigr]=j_{\cv(\unito,X)}.
  \]
  The image of \(*\) under the composite in the left hand side is
  \(\cv(\unito,\gamma^{-1}(j_X))=\cv(\unito,1_X)=1_{\cv(\unito,X)}\),
  which is precisely \(j_{\cv(\unito,X)}(*)\).

  \paragraph{CF2.} We must prove the following equation:
  \begin{align*}
    \bigl[
    \cv(\unito,X)&\rto{\cv(\unito,i_X)}\cv(\unito,\und\cv(\unito,X))
    \\
    &\rto[\hphantom{\cv(\unito,i_X)}]{\gamma^{-1}}\cv(\unito,X)
    \\
    &\rto[\hphantom{\cv(\unito,i_X)}]{\cv(\unito,-)}\Set(\cv(\unito,\unito),\cv(\unito,X))
    \\
    &\rto[\hphantom{\cv(\unito,i_X)}]{\Set(e^0,1)}\Set(\{*\},\cv(\unito,X))
    \bigr]=i_{\cv(\unito,X)}.
  \end{align*}
  By \corref{cor-gamma-C-1i-1} the left hand side is equal to 
  \[
  \bigl[ \cv(\unito,X) \rto{\cv(\unito,-)}
  \Set(\cv(\unito,\unito),\cv(\unito,X)) \rto{\Set(e^0,1)}
  \Set(\{*\},\cv(\unito,X)) \bigr],
  \]
  and so it maps an element \(f\in\cv(\unito,X)\) to the function
  \(\{*\}\to\cv(\unito,X)\), \(*\mapsto f\), which is precisely
  \(i_{\cv(\unito,X)}(f)\).

  \paragraph{CF3.} We must prove that the exterior of the following
  diagram commutes:
  \[
  \begin{xy}
    (-45,36)*+{\cv(\unito,\und\cv(Y,Z))}="1";
    (45,36)*+{\cv(\unito,\und\cv(\und\cv(X,Y),\und\cv(X,Z)))}="2";
    (45,18)*+{\cv(\und\cv(X,Y),\und\cv(X,Z))}="4";
    (-45,18)*+{\cv(Y,Z)}="3";
    (45,0)*+{\Set(\cv(\unito,\und\cv(X,Y)),\cv(\unito,\und\cv(X,Z)))}="6";
    (-45,0)*+{\Set(\cv(\unito,Y),\cv(\unito,Z))}="5";
    (45,-18)*+{\Set(\cv(\unito,\und\cv(X,Y)),\cv(X,Z))}="8";
    (-45,-18)*+{\Set(\Set(\cv(\unito,X),\cv(\unito,Y)),\Set(\cv(\unito,X),\cv(\unito,Z)))}="7";
    (-45,-36)*+{\Set(\cv(X,Y),\Set(\cv(\unito,X),\cv(\unito,Z)))}="9";
    (45,-36)*+{\Set(\cv(\unito,\und\cv(X,Y)),\Set(\cv(\unito,X),\cv(\unito,Z)))}="10";
    {\ar@{->}^-{\cv(\unito,L^X_{YZ})} "1";"2"};
    {\ar@{->}_-{\gamma^{-1}} "1";"3"};
    {\ar@{->}^-{\gamma^{-1}} "2";"4"};
    {\ar@{->}^-{\und\cv(X,-)} "3";"4"};
    {\ar@{->}_-{\cv(\unito,-)} "3";"5"};
    {\ar@{->}^-{\cv(\unito,-)} "4";"6"};
    {\ar@{->}_-{L^{\cv(\unito,X)}_{\cv(\unito,Y),\cv(\unito,Z)}} "5";"7"};
    {\ar@{->}^-{\Set(1,\gamma^{-1})} "6";"8"};
    {\ar@{->}_-{\Set(\cv(\unito,-),1)} "7";"9"};
    {\ar@{->}^-{\Set(1,\cv(\unito,-))} "8";"10"};
    {\ar@{->}^-{\Set(\gamma^{-1},1)} "9";"10"};
  \end{xy}
  \]
  The upper square commutes by
  \propref{prop-gamma-C1L-undC-gamma}. Let us prove that so does the
  remaining region. Taking an element \(f\in\cv(Y,Z)\) and tracing it
  along the top-right path we obtain the function
  \begin{align*}
    \cv(\unito,\und\cv(X,Y))&\to\Set(\cv(\unito,X),\cv(\unito,Z)),
    \\
    g&\mapsto\bigl(h\mapsto
    h\cdot\gamma^{-1}(g\cdot\und\cv(1,f))\bigr),
    \\
    \intertext{whereas pushing \(f\) along the left-bottom path yields the function}
    \cv(\unito,\und\cv(X,Y))&\to\Set(\cv(\unito,X),\cv(\unito,Z)),
    \\
    g&\mapsto\bigl(h\mapsto h\cdot\gamma^{-1}(g)\cdot f\bigr).
  \end{align*}
  These two functions are equal by
  \propref{prop-gamma-fg-gamma-f-C1g}. The proposition is proven.
\end{proof}

\begin{definition}
  Let
  \(\Phi=(\phi,\hat\phi,\phi^0),\Psi=(\psi,\hat\psi,\psi^0):\cc\to\cd\)
  be closed functors. A \emph{closed natural transformation}
  \(\eta:\Phi\to\Psi:\cc\to\cd\) is a natural transformation
  \(\eta:\phi\to \psi:\cc\to\cd\) satisfying the following axioms.
  \begin{list}{}{%
      \setlength{\labelwidth}{3.5em}
      \setlength{\leftmargin}{4em}}
  \item[CN1.] The following equation holds true:
    \begin{equation*}
      \bigl[
      \unito\rto{\phi^0}\phi\unito\rto{\eta_\unito}\psi\unito
      \bigr]=\psi^0.
    \end{equation*}

  \item[CN2.] The following diagram commutes:
    \begin{equation*}
      \begin{xy}
        (-40,9)*+{\phi\und\cc(X,Y)}="1";
        (40,9)*+{\und\cd(\phi X,\phi Y)}="2";
        (40,-9)*+{\und\cd(\phi X,\psi Y)}="3";
        (-40,-9)*+{\psi\und\cc(X,Y)}="4";
        (0,-9)*+{\und\cd(\psi X,\psi Y)}="5";
        {\ar@{->}^-{\hat\phi} "1";"2"};
        {\ar@{->}_-{\eta_{\und\cc(X,Y)}} "1";"4"};
        {\ar@{->}^-{\und\cd(1,\eta_Y)} "2";"3"};
        {\ar@{->}^-{\und\cd(\eta_X,1)} "5";"3"};
        {\ar@{->}^-{\hat\psi} "4";"5"};
      \end{xy}
    \end{equation*}
  \end{list}
\end{definition}

Closed categories, closed functors, and closed natural transformations
form a 2\n-category \cite[Theorem~4.2]{EK}, which we shall denote by
\(\ClCat\). The composite of closed functors
\(\Phi=(\phi,\hat\phi,\phi^0):\cc\to\cd\) and
\(\Psi=(\psi,\hat\psi,\psi^0):\cd\to\ce\) is defined to be
\(\Chi=(\chi,\hat\chi,\chi^0):\cc\to\ce\), where:
\begin{itemize}
\item \(\chi\) is the composite \(\cc\rto\phi\cd\rto\psi\ce\);
\item \(\hat\chi\) is the composite
  \(\psi\phi\und\cc(X,Y)\rto{\psi\hat\phi}\psi\und\cd(\phi X,\phi
  Y)\rto{\hat\psi}\und\ce(\psi\phi X,\psi\phi Y)\);
\item \(\chi^0\) is the composite
  \(\unito\rto{\psi^0}\psi\unito\rto{\psi\phi^0}\psi\phi\unito\).
\end{itemize}
Compositions of closed natural transformations are defined in the
usual way.

We can enrich in closed categories. Below we recall some enriched
category theory for closed categories mainly following
\cite[Section~5]{EK}.

\begin{definition}
  Let \(\cv\) be a closed category. A \emph{\(\cv\)\n-category} \(\ca\) consists
  of the following data:
  \begin{itemize}
  \item a set \(\Ob\ca\) of objects;
  \item for each \(X,Y\in\Ob\ca\), an object \(\ca(X,Y)\) of \(\cv\);
  \item for each \(X\in\Ob\ca\), a morphism \(j_X:\unito\to\ca(X,X)\) in
    \(\cv\);
  \item for each \(X,Y,Z\in\Ob\ca\), a morphism
    \(L^X_{YZ}:\ca(Y,Z)\to\und\cv(\ca(X,Y),\ca(X,Z))\) in \(\cv\).
  \end{itemize}
  These data are to satisfy axioms \cite[VC1--VC3]{EK}.  If \(\ca\)
  and \(\cb\) are \(\cv\)\n-categories, a \emph{\(\cv\)\n-functor}
  \(F:\ca\to\cb\) consists of the following data:
  \begin{itemize}
  \item a function \(\Ob F:\Ob\ca\to\Ob\cb\), \(X\mapsto FX\);
  \item for each \(X,Y\in\Ob\ca\), a morphism
    \(F=F_{XY}:\ca(X,Y)\to\cb(FX,FY)\) in \(\cv\).
  \end{itemize}
  These data are subject to axioms \cite[VF1--VF2]{EK}.  
\end{definition}

\begin{example}\label{ex-V-LX}
  By \cite[Theorem~5.2]{EK} a closed category \(\cv\) gives rise to a
  category \(\und\cv\) if we take the objects of \(\und\cv\) to be
  those of \(\cv\), take \(\und\cv(X,Y)\) to be the internal
  \(\Hom\)\n-object, and take for \(j\) and \(L\) those of the closed
  category \(\cv\). Furthermore, if \(\ca\) is a \(\cv\)\n-category
  and \(X\) is an object of \(\ca\), then we get a \(\cv\)\n-functor
  \(L^X:\ca\to\und\cv\) if we take \(L^XY=\ca(X,Y)\) and
  \((L^X)_{YZ}=L^X_{YZ}\). In particular, for each \(X\in\Ob\cv\),
  there is a \(\cv\)\n-functor \(L^X:\und\cv\to\und\cv\) such that
  \(L^XY=\und\cv(X,Y)\) and \((L^X)_{YZ}=L^X_{YZ}\).
\end{example}

There is also a notion of \(\cv\)\n-natural transformation. We recall
it in a particular case, namely for \(\cv\)\n-functors
\(\ca\to\und\cv\).

\begin{definition}
  Let \(F,G:\ca\to\und\cv\) be \(\cv\)-functors. A
  \emph{\(\cv\)\n-natural transformation} \(\alpha:F\to
  G:\ca\to\und\cv\) is a collection of morphisms \(\alpha_X:FX\to GX\)
  in \(\cv\), for each \(X\in\Ob\ca\), such that the diagram
  \[
  \begin{xy}
    (-20,9)*+{\ca(X,Y)}="1";
    (20,9)*+{\und\cv(FX,FY)}="2";
    (-20,-9)*+{\und\cv(GX,GY)}="3";
    (20,-9)*+{\und\cv(FX,GY)}="4";
    {\ar@{->}^-{F_{XY}} "1";"2"};
    {\ar@{->}_-{G_{XY}} "1";"3"};
    {\ar@{->}^-{\und\cv(1,\alpha_Y)} "2";"4"};
    {\ar@{->}^-{\und\cv(\alpha_X,1)} "3";"4"}
  \end{xy}
  \]
  commutes, for each \(X,Y\in\Ob\ca\).
\end{definition}

\begin{example}\label{exa-Lf}
  By \cite[Proposition~8.4]{EK} if \(f\in\cv(X,Y)\), the morphisms
  \[
  \und\cv(f,1):\und\cv(Y,Z)\to\und\cv(X,Z), \quad Z\in\Ob\cv,
  \]
  are components of a \(\cv\)\n-natural transformation \(L^f:L^Y\to
  L^X:\und\cv\to\und\cv\).
\end{example}

By \cite[Theorem~10.2]{EK} \(\cv\)\n-categories, \(\cv\)\n-functors,
and \(\cv\)\n-natural transformations form a 2-category, which we
shall denote by \(\VCat\).

\begin{proposition}[{\cite[Proposition~6.1]{EK}}]
  If \(\Phi=(\phi,\hat\phi,\phi^0):\cv\to\cw\) is a closed functor and
  \(\ca\) is a \(\cv\)\n-category, the following data define a
  \(\cw\)\n-category \(\Phi_*\ca\):
  \begin{itemize}
  \item \(\Ob\Phi_*\ca=\Ob\ca\);
  \item \((\Phi_*\ca)(X,Y)=\phi\ca(X,Y)\);
  \item \(j_X=\bigl[\unito\rto{\phi^0}\phi\unito\rto{\phi
      j_X}\phi\ca(X,X)\bigr]\);
  \item \(L^X_{YZ}=\bigl[\phi\ca(Y,Z)\rto{\phi
      L^X_{YZ}}\phi\und\cv(\ca(X,Y),\ca(X,Z))\rto{\hat\phi}\und\cw(\phi\ca(X,Y),\phi\ca(X,Z))\bigr]\).
  \end{itemize}
\end{proposition}

\begin{example}
  Let us study the effect of the closed functor \(E\) from
  \propref{prop-E-V1?} on \(\cv\)\n-categories. Let \(\ca\) be a
  \(\cv\)\n-category. Then the ordinary category \(E_*\ca\) has the
  same set of objects as \(\ca\) and its \(\Hom\)\n-sets are
  \((E_*\ca)(X,Y)=\cv(\unito,\ca(X,Y))\). The morphism \(j_X\) for the
  category \(E_*\ca\) is given by the composite
  \[
  \{*\}\rto{e^0}\cv(\unito,\unito)\rto{\cv(\unito,j_X)}\cv(\unito,\ca(X,X)),
  \]
  i.e., \(1_X\in(E_*\ca)(X,X)\) identifies with \(j_X\). The morphism
  \(L^X_{YZ}\) for the category \(E_*\ca\) is given by the composite
  \begin{align*}
    \cv(\unito,\ca(Y,Z))&\rto{\cv(\unito,L^X)}\cv(\unito,\und\cv(\ca(X,Y),\ca(X,Z)))
    \\
    &\rto[\hphantom{\cv(\unito,L^X)}]{\gamma^{-1}}\cv(\ca(X,Y),\ca(X,Z))
    \\
    &\rto[\hphantom{\cv(\unito,L^X)}]{\cv(\unito,-)}
    \Set(\ca(\unito,\ca(X,Y)),\cv(\unito,\ca(X,Z))).
  \end{align*}
  It follows that composition in \(E_*\ca\) is given by
  \[
  \cv(\unito,\ca(X,Y))\times\cv(\unito,\ca(Y,Z))\to\cv(\unito,\ca(X,Z)),\quad(f,g)\mapsto
  f\cdot\gamma^{-1}(g\cdot L^X_{YZ}).
  \]
\end{example}

\begin{proposition}\label{prop-gamma-iso-V-EundV}
  Let \(\cv\) be a closed category. There is an isomorphism of
  categories \(\gamma:\cv\to E_*\und\cv\) that is identical on objects
  and is given by the bijections
  \(\gamma:\cv(X,Y)\to\cv(\unito,\und\cv(X,Y))\) on morphisms.
\end{proposition}

\begin{proof}
  For each \(X\in\Ob\cv\), we have \(\gamma(1_X)=j_X\), so \(\gamma\)
  preserves identities. Let us check that it also preserves
  composition. For each \(f\in\cv(X,Y)\), \(g\in\cv(Y,Z)\), we have
  \(\gamma(f)\cdot\gamma(g)=\gamma(f)\cdot\gamma^{-1}(\gamma(g)\cdot
  L^X_{YZ})\). By \propref{prop-gamma-C1L-undC-gamma},
  \(\gamma(g)\cdot L^X_{YZ}=\gamma(\und\cv(1,g))\), therefore
  \(\gamma(f)\cdot\gamma(g)=\gamma(f)\cdot\und\cv(1,g)=\gamma(f\cdot
  g)\) by \propref{prop-gamma-fg-gamma-f-C1g}. The proposition is
  proven.
\end{proof}

\begin{theorem}\label{thm-cl-cat-iso-EK}
  Every closed category is isomorphic to a closed category in the
  sense of Eilenberg and Kelly.
\end{theorem}

More precisely, for every closed category \(\cv\) in the sense of
\defref{def:closed-category} there is a closed category \(\cw\) in the
sense of Eilenberg and Kelly such that \(\cw\), when viewed as a
closed category in the sense of \defref{def:closed-category}, is
isomorphic as a closed category to \(\cv\).

\begin{proof}
  Let \(\cv\) be a closed category. Take \(\cw=E_*\und\cv\). The
  isomorphism \(\gamma\) from \propref{prop-gamma-iso-V-EundV} allows
  us to translate the structure of a closed category from \(\cv\) to
  \(\cw\). Thus the unit object of \(\cw\) is that of \(\cv\), the
  internal \(\Hom\)\n-functor is given by the composite
  \[
  \und\cw(-,-)=\bigl[
  \cw^\op\times\cw\rto{(\gamma^\op\times\gamma)^{-1}}\cv^\op\times\cv\rto{\und\cv(-,-)}\cv\rto{\gamma}\cw
  \bigr].
  \]
  In particular, \(\und\cw(X,Y)=\und\cv(X,Y)\) for each pair of
  objects \(X\) and \(Y\). The transformations \(i_X\), \(j_X\),
  \(L^X_{YZ}\) for \(\cw\) are just \(\gamma(i_X)\), \(\gamma(j_X)\),
  \(\gamma(L^X_{YZ})\) respectively. The category \(\cw\) admits a
  functor \(W:\cw\to\Set\) such that the diagram
  \[
  \begin{xy}
    (-17,9)*+{\cw^\op\times\cw}="1"; (17,9)*+{\cw}="2";
    (17,-9)*+{\Set}="3"; {\ar@{->}^-{\und\cw(-,-)} "1";"2"};
    {\ar@{->}_-{\cw(-,-)} "1";"3"}; {\ar@{->}^-{W} "2";"3"};
  \end{xy}
  \]
  commutes, namely \(W=\bigl[ \cw\rto{\gamma^{-1}}\cv\rto{E}\Set
  \bigr]\).  The commutativity on objects is obvious. Let us check
  that it also holds on morphisms. Let \(f\in\cw(X,Y)\),
  \(h\in\cw(U,V)\); i.e., suppose that \(f:\unito\to\und\cv(X,Y)\) and
  \(h:\unito\to\und\cv(U,V)\) are morphisms in \(\cv\). Then the map
  \(\cw(f,g):\cw(Y,U)\to\cw(X,V)\) is given by \(g \mapsto f \cdot g
  \cdot h\), where the composition is taken in \(\cw\). We must show
  that it is equal to the map
  \[
  \cv(\unito,\und\cv(\gamma^{-1}(f),\gamma^{-1}(h))) :
  \cv(\unito,\und\cv(Y,U))\to\cv(\unito,\und\cv(X,V)), \quad g \mapsto
  g \cdot \und\cv(\gamma^{-1}(f),\gamma^{-1}(h)).
  \]
  We have:
  \begin{alignat*}{3}
    g \cdot \und\cv(\gamma^{-1}(f),\gamma^{-1}(h)) &=
    \gamma(\gamma^{-1}(g)) \cdot \und\cv(\gamma^{-1}(f),1) \cdot
    \und\cv(1,\gamma^{-1}(h)) & \quad &
    \textup{(functoriality of \(\und\cv(-,-)\))}
    \\
    &= \gamma(\gamma^{-1}(f)\cdot\gamma^{-1}(g)) \cdot
    \und\cv(1,\gamma^{-1}(h)) & \quad &
    \textup{(\propref{prop-gamma-fg-gamma-f-C1g})}
    \\
    &= \gamma(\gamma^{-1}(f)\cdot\gamma^{-1}(g)\cdot\gamma^{-1}(h))
    & \quad & \textup{(\propref{prop-gamma-fg-gamma-f-C1g})}
    \\
    &= f\cdot g\cdot h, & \quad & \textup{(\propref{prop-gamma-iso-V-EundV})}
  \end{alignat*}
  hence the assertion. The functor \(W\) also satisfies the axiom
  CC5'.  Indeed, we need to show that
  \[
  Wi_{\und\cw(X,X)}=\cv(\unito,i_{\und\cv(X,X)}):
  \cv(\unito,\und\cv(X,X))\to\cv(\unito,\und\cv(\unito,\und\cv(X,X)))
  \]
  maps \(j_X\in\cv(\unito,\und\cv(X,X))\) to
  \(\gamma(j_X)\in\cv(\unito,\und\cv(\unito,\und\cv(X,X)))\). In other
  words, we need to show that the diagram
  \[
  \begin{xy}
    (-17,9)*+{\unito}="1"; (17,9)*+{\und\cv(X,X)}="2";
    (-17,-9)*+{\und\cv(\unito,\unito)}="3";
    (17,-9)*+{\und\cv(\unito,\und\cv(X,X))}="4"; {\ar@{->}^-{j_X}
      "1";"2"}; {\ar@{->}_-{j_\unito} "1";"3"};
    {\ar@{->}^-{i_{\und\cv(X,X)}} "2";"4"};
    {\ar@{->}^-{\und\cv(1,j_X)} "3";"4"};
  \end{xy}
  \]
  commutes. However
  \(j_\unito=i_\unito:\unito\to\und\cv(\unito,\unito)\) by
  \propref{prop-j1-i1}, so the above diagram is commutative by the
  naturality of \(i\). The theorem is proven.
\end{proof}

Finally, let us recall from \cite{EK} the representation theorem for
\(\cv\)\n-functors \(\ca\to\und\cv\).

\begin{proposition}[{\cite[Corollary~8.7]{EK}}]\label{prop-repr}
  Suppose that \(\cv\) is a closed category in the sense of Eilenberg
  and Kelly; i.e., it is equipped with a functor \(V:\cv\to\Set\)
  satisfying the axioms CC0 and CC5'. Let \(T:\ca\to\und\cv\) be a
  \(\cv\)\n-functor, and let \(W\) be an object of \(\ca\). Then the
  map\footnote{It is denoted by \(\Gamma'\) in
    \cite[Corollary~8.7]{EK}.}
  \[
  \Gamma:\VCat(\ca,\und\cv)(L^W,T)\to VTW, \quad p\mapsto (Vp_W)1_W,
  \]
  is a bijection.
\end{proposition}

\begin{example}\label{ex-Lf-repr}
  For each \(f\in VL^XY=V\und\cv(X,Y)=\cv(X,Y)\), the
  \(\cv\)\n-natural transformation \(L^f:L^Y\to
  L^X:\und\cv\to\und\cv\) from \exaref{exa-Lf} is uniquely determined
  by the condition \((V(L^f)_Y)1_Y=f\).
\end{example}

\section{Closed multicategories}
\label{sec:clos-mult}

We begin by briefly recalling the notions of multicategory,
multifunctor, and multinatural transformation. The reader is referred
to the excellent book by Leinster \cite{Leinster} or to
\cite[Chapter~3]{BLM} for a more elaborate introduction to
multicategories.

\begin{definition}
  A \emph{multigraph} \(\mcC\) is a set \(\Ob\mcC\), whose elements
  are called \emph{objects} of \(\mcC\), together with a
  set \(\mcC(X_1,\dots,X_n;Y)\) for each \(n\in\NN\) and
  \(X_1,\dots,X_n,Y\in\Ob\mcC\). Elements of \(\mcC(X_1,\dots,X_n;Y)\)
  are called \emph{morphisms} and written as \(X_1,\dots,X_n\to
  Y\). If \(n=0\), elements of \(\mcC(;Y)\) are written as \(()\to
  Y\). A morphism of multigraphs \(F:\mcC\to\mcD\) consists of a
  function \(\Ob F:\Ob\mcC\to\Ob\mcD\), \(X\mapsto FX\), and functions
  \[
  F=F_{X_1,\dots,X_n;Y}:\mcC(X_1,\dots,X_n;Y)\to\mcD(FX_1,\dots,FX_n;FY),
  \quad f\mapsto Ff,
  \]
  for each \(n\in\NN\) and \(X_1,\dots,X_n,Y\in\Ob\mcC\).
\end{definition}

\begin{definition}
  A \emph{multicategory} \(\mcC\) consists of the following data:
  \begin{itemize}

  \item a multigraph \(\mcC\);

  \item for each \(n,k_1,\dots,k_n\in\NN\) and
    \(X_{ij},Y_i,Z\in\Ob\mcC\), \(1\le i\le n\), \(1\le j\le k_i\), a
    function
    \[
    \prod_{i=1}^n\mcC(X_{i1},\dots,X_{ik_i};Y_i)\times\mcC(Y_1,\dots,Y_n;Z)\to
    \mcC(X_{11},\dots,X_{1k_1},\dots,X_{n1},\dots,X_{nk_n};Z),
    \]
    called \emph{composition} and written \((f_1,\dots,f_n,g)\mapsto
    (f_1,\dots,f_n)\cdot g\);

  \item for each \(X\in\Ob\mcC\), an element \(1^\mcC_X\in\mcC(X;X)\),
    called the \emph{identity} of \(X\).
  \end{itemize}
  These data are subject to the obvious associativity and identity
  axioms.
\end{definition}

\begin{example}\label{ex:monoidal-cat-multicat}
  A strict monoidal category \(\cc\) gives rise to a
  multicategory \(\wh\cc\) as follows:
  \begin{itemize}
  \item \(\Ob\wh\cc=\Ob\cc\);
  \item for each \(n\in\NN\) and \(X_1\), \dots, \(X_n\),
    \(Y\in\Ob\cc\), \(\wh\cc(X_1,\dots,X_n;Y)=\cc(X_1\tens\dots\tens
    X_n,Y)\); in particular \(\wh\cc(;Y)=\cc(\unito,Y)\), where
    \(\unito\) is the unit object of \(\cc\);
  \item for each \(n\), \(k_1\), \dots, \(k_n\in\NN\) and \(X_{ij}\),
    \(Y_i\), \(Z\in\Ob\cc\), \(1\le i\le n\), \(1\le j\le k_i\), the
    composition map
    \begin{multline*}
      \prod_{i=1}^n\cc(X_{i1}\tens\dots\tens X_{ik_i},Y_i)\times
      \cc(Y_1\tens\dots\tens Y_n,Z)\to \cc(X_{11}\tens\dots\tens
      X_{1k_1}\tens\dots\tens X_{n1}\tens\dots\tens X_{nk_n},Z)
    \end{multline*}
    is given by \((f_1,\dots,f_n,g)\mapsto(f_1\tens\dots\tens f_n)\cdot
    g\);

  \item for each \(X\in\Ob\cc\),
    \(1_X^{\wh\cc}=1_X^\cc\in\wh\cc(X;X)=\cc(X,X)\).
  \end{itemize}
\end{example}

\begin{definition}
  Let \(\mcC\) and \(\mcD\) be multicategories. A \emph{multifunctor}
  \(F:\mcC\to\mcD\) is a morphism of the underlying multigraphs that
  preserves composition and identities.
\end{definition}

\begin{definition}
  Suppose that \(F,G:\mcC\to\mcD\) are multifunctors. A
  \emph{multinatural transformation} \(r:F\to G:\mcC\to\mcD\) is a
  family of morphisms \(r_X\in\mcD(FX;GX)\), \(X\in\Ob\mcC\), such
  that
  \[
  Ff\cdot r_Y=(r_{X_1},\dots,r_{X_n})\cdot Gf,
  \]
  for each \(f\in\mcC(X_1,\dots,X_n;Y)\).
\end{definition}

Multicategories, multifunctors, and multinatural transformations form
a 2\n-category, which we shall denote by \(\Multicat\).

\begin{definition}[{\cite[Definition~4.7]{BLM}}]
  A multicategory \(\mcC\) is called \emph{closed} if for each
  \(m\in\NN\) and \(X_1,\dots,X_m,Z\in\Ob\mcC\) there exist an object
  \(\und\mcC(X_1,\dots,X_m;Z)\), called \emph{internal
    \(\Hom\)\n-object}, and an \emph{evaluation} morphism
  \[
  \ev^\mcC=\ev^\mcC_{X_1,\dots,X_m;Z}:X_1,\dots,X_m,\und\mcC(X_1,\dots,X_m;Z)\to
  Z
  \]
  such that, for each \(Y_1,\dots,Y_n\in\Ob\mcC\), the function
  \[
  \varphi^\mcC=\varphi^\mcC_{X_1,\dots,X_m;Y_1,\dots,Y_n;Z}:
  \mcC(Y_1,\dots,Y_n;\und\mcC(X_1,\dots,X_m;Z))\to\mcC(X_1,\dots,X_m,Y_1,\dots,Y_n;Z)
  \]
  that sends a morphism
  \(f:Y_1,\dots,Y_n\to\und\mcC(X_1,\dots,X_m;Z)\) to the composite
  \[
  X_1,\dots,X_m,Y_1,\dots,Y_n\rto{1^\mcC_{X_1},\dots,1^\mcC_{X_m},f}X_1,\dots,X_m,\und\mcC(X_1,\dots,X_m;Z)
  \rto{\ev^\mcC_{X_1,\dots,X_m;Z}}Z
  \]
  is bijective. Let \(\ClMulticat\) denote the full 2\n-subcategory of
  \(\Multicat\) whose objects are closed multicategories.
\end{definition}

\begin{remark}
  Notice that for \(m=0\) an object \(\und\mcC(;Z)\) and a morphism
  \(\ev^\mcC_{;Z}\) with the required property always exist. Namely,
  we may (and we shall) always take \(\und\mcC(;Z)=Z\) and
  \(\ev^\mcC_{;Z}=1^\mcC_Z:Z\to Z\). With these choices
  \(\varphi^\mcC_{;Y_1,\dots,Y_n;Z}:\mcC(Y_1,\dots,Y_n;Z)
  \to\mcC(Y_1,\dots,Y_n;Z)\) is the identity map.
\end{remark}

\begin{example}
  Let \(\cc\) be a strict monoidal category, and let \(\wh\cc\) be the
  associated multicategory, see \exaref{ex:monoidal-cat-multicat}.  It
  is easy to see that the multicategory \(\wh\cc\) is closed if and
  only if \(\cc\) is closed as a monoidal category.
\end{example}

\begin{proposition}\label{prop-undCXZ-implies-closedness}
  Suppose that for each pair of objects \(X,Z\in\Ob\mcC\) there exist
  an object \(\und\mcC(X;Z)\) and a morphism
  \(\ev^\mcC_{X;Z}:X,\und\mcC(X;Z)\to Z\) of \(\mcC\) such that the
  function \(\varphi^\mcC_{X;Y_1,\dots,Y_n;Z}\) is a bijection, for
  each finite sequence \(Y_1,\dots,Y_n\) of objects of \(\mcC\). Then
  \(\mcC\) is a closed multicategory.
\end{proposition}

\begin{proof}
  Define internal \(\Hom\)\n-objects \(\und\mcC(X_1,\dots,X_m;Z)\) and
  evaluations
  \[
  \ev^\mcC_{X_1,\dots,X_m;Z}:X_1,\dots,X_m,\und\mcC(X_1,\dots,X_m;Z)\to
  Z
  \]
  by induction on \(m\). For \(m=0\) choose \(\und\mcC(;Z)=Z\) and
  \(\ev^\mcC_{;Z}=1^\mcC_Z:Z\to Z\) as explained above. For \(m=1\) we
  are already given \(\und\mcC(X;Z)\) and \(\ev^\mcC_{X;Z}\). Assume
  that we have defined \(\und\mcC(X_1,\dots,X_k;Z)\) and
  \(\ev^\mcC_{X_1,\dots,X_k;Z}\) for each \(k<m\), and that the
  function
  \[
  \varphi^\mcC_{X_1,\dots,X_k;Y_1,\dots,Y_n;Z}:
  \mcC(Y_1,\dots,Y_n;\und\mcC(X_1,\dots,X_k;Z))\to
  \mcC(X_1,\dots,X_k,Y_1,\dots,Y_n;Z)
  \]
  is a bijection, for each \(k<m\) and for each finite sequence
  \(Y_1,\dots,Y_n\) of objects of \(\mcC\). For \(X_1,\dots,X_m\),
  \(Z\in\Ob\mcC\) define
  \[
  \und\mcC(X_1,\dots,X_m;Z)\defeq\und\mcC(X_m;\und\mcC(X_1,\dots,X_{m-1};Z)).
  \]
  The evaluation morphism \(\ev^\mcC_{X_1,\dots,X_m;Z}\) is given by
  the composite
  \[
  \begin{xy}
    (0,15)*+{X_1,\dots,X_m,\und\mcC(X_m;\und\mcC(X_1,\dots,X_{m-1};Z))}="1";
    (0,0)*+{X_1,\dots,X_{m-1},\und\mcC(X_1,\dots,X_{m-1};Z)}="2";
    (0,-15)*+{Z.}="3";
    {\ar@{->}^-{1^\mcC_{X_1},\dots,1^\mcC_{X_{m-1}},\ev^\mcC_{X_m;\und\mcC(X_1,\dots,X_{m-1};Z)}}
      "1";"2"}; {\ar@{->}^-{\ev^\mcC_{X_1,\dots,X_{m-1};Z}} "2";"3"};
  \end{xy}
  \]
  It is easy to see that with these choices the function
  \(\varphi^\mcC_{X_1,\dots,X_m;Y_1,\dots,Y_n;Z}\) decomposes as
  \[
  \begin{xy}
    (0,15)*+{\mcC(Y_1,\dots,Y_n;\und\mcC(X_1,\dots,X_m;Z))}="1";
    (0,0)*+{\mcC(X_m,Y_1,\dots,Y_n;\und\mcC(X_1,\dots,X_{m-1};Z))}="2";
    (0,-15)*+{\mcC(X_1,\dots,X_m,Y_1,\dots,Y_n;Z),}="3";
    {\ar@{->}^-{\varphi^\mcC_{X_m;Y_1,\dots,Y_n;\und\mcC(X_1,\dots,X_{m-1};Z)}}_-\wr
      "1";"2"};
    {\ar@{->}^-{\varphi^\mcC_{X_1,\dots,X_{m-1};X_m,Y_1,\dots,Y_n;Z}}_-\wr
      "2";"3"};
  \end{xy}
  \]
  hence it is a bijection, and the induction goes through.
\end{proof}

\begin{notation}
  For each morphism \(f:X_1,\dots,X_n\to Y\) with \(n\ge1\), denote by
  \(\langle f\rangle\) the morphism
  \((\varphi_{X_1;X_2,\dots,X_n;Z})^{-1}(f):X_2,\dots,X_n\to\und\mcC(X_1;Y)\).
  In other words, \(\langle f\rangle\) is uniquely determined by the equation
  \[
  \bigl[ X_1,X_2,\dots,X_n\rto{1^\mcC_{X_1},\langle
    f\rangle}X_1,\und\mcC(X_1;Y)\rto{\ev^\mcC_{X_1;Y}}Y \bigr]=f.
  \]
\end{notation}

Clearly we can enrich in multicategories. We leave it as an easy
exercise for the reader to spell out the definitions of categories and
functors enriched in a multicategory \(\mcV\).

\begin{proposition}
  A closed multicategory \(\mcC\) gives rise to a \(\mcC\)\n-category
  \(\und\mcC\) as follows. The objects of \(\und\mcC\) are those of
  \(\mcC\). For each pair \(X,Y\in\Ob\mcC\), the \(\Hom\)\n-object
  \(\und\mcC(X;Y)\) is the internal \(\Hom\)\n-object of \(\mcC\). For
  each \(X,Y,Z\in\Ob\mcC\), the composition morphism
  \(\mu_{\und\mcC}:\und\mcC(X;Y),\und\mcC(Y;Z)\to\und\mcC(X;Z)\) is
  uniquely determined by requiring the commutativity in the diagram
  \[
  \begin{xy}
    (-22,9)*+{X,\und\mcC(X;Y),\und\mcC(Y;Z)}="1";
    (22,9)*+{X,\und\mcC(X;Z)}="2"; (-22,-9)*+{Y,\und\mcC(Y;Z)}="3";
    (22,-9)*+{Z}="4"; {\ar@{->}^-{1^\mcC_X,\mu_{\und\mcC}} "1";"2"};
    {\ar@{->}_-{\ev^\mcC_{X;Y},1^\mcC_{\und\mcC(Y;Z)}} "1";"3"};
    {\ar@{->}^-{\ev^\mcC_{X;Z}} "2";"4"}; {\ar@{->}^-{\ev^\mcC_{Y;Z}}
      "3";"4"};
  \end{xy}
  \]
  The identity of an object \(X\in\Ob\mcC\) is
  \(1^{\und\mcC}_X=\langle 1^\mcC_X\rangle:()\to\und\mcC(X;X)\).
\end{proposition}

\begin{proof}
  The proof is similar to that for a closed monoidal category.
\end{proof}
\begin{notation}
  For each morphism \(f:X_1,\dots,X_n\to Y\) and object \(Z\) of a
  closed multicategory \(\mcC\), there exists a unique morphism
  \(\und\mcC(f;Z):\und\mcC(Y;Z)\to\und\mcC(X_1,\dots,X_n;Z)\) such
  that the diagram
  \[
  \begin{xy}
    (-37,9)*+{X_1,\dots,X_n,\und\mcC(Y;Z)}="1";
    (37,9)*+{X_1,\dots,X_n,\und\mcC(X_1,\dots,X_n;Z)}="2";
    (-37,-9)*+{Y,\und\mcC(Y;Z)}="3"; (37,-9)*+{Z}="4";
    {\ar@{->}^-{1^\mcC_{X_1},\dots,1^\mcC_{X_n},\und\mcC(f;Z)}
      "1";"2"}; {\ar@{->}_-{f,1^\mcC_{\und\mcC(Y;Z)}} "1";"3"};
    {\ar@{->}^-{\ev^\mcC_{X_1,\dots,X_n;Z}} "2";"4"};
    {\ar@{->}^-{\ev^\mcC_{Y;Z}} "3";"4"};
  \end{xy}
  \]
  in \(\mcC\) is commutative. In particular, if \(n=0\), then
  \(\und\mcC(f;Z)=(f,1^\mcC_{\und\mcC(Y;Z)})\cdot\ev^\mcC_{Y;Z}\). If
  \(n=1\), then \(\und\mcC(f;Z)=\langle(f,
  1^\mcC_{\und\mcC(Y;Z)})\cdot\ev^\mcC_{Y;Z}\rangle\).  For each
  sequence of morphisms \(f_1:X_1\to Y_1\), \dots, \(f_n:X_n\to Y_n\)
  in \(\mcC\) there is a unique morphism \(\und\mcC(f_1,\dots,f_n;Z):
  \und\mcC(Y_1,\dots,Y_n;Z) \to \und\mcC(X_1,\dots,X_n;Z)\) such that
  the diagram
  \[
  \begin{xy}
    (-45,9)*+{X_1,\dots,X_n,\und\mcC(Y_1,\dots,Y_n;Z)}="1";
    (45,9)*+{X_1,\dots,X_n,\und\mcC(X_1,\dots,X_n;Z)}="2";
    (-45,-9)*+{Y_1,\dots,Y_n,\und\mcC(Y_1,\dots,Y_n;Z)}="3";
    (45,-9)*+{Z}="4";
    {\ar@{->}^-{1^\mcC_{X_1},\dots,1^\mcC_{X_n},\und\mcC(f_1,\dots,f_n;Z)}
      "1";"2"};
    {\ar@{->}_-{f_1,\dots,f_n,1^\mcC_{\und\mcC(Y_1,\dots,Y_n;Z)}}
      "1";"3"};
    {\ar@{->}^-{\ev^\mcC_{X_1,\dots,X_n;Z}} "2";"4"};
    {\ar@{->}^-{\ev^\mcC_{Y_1,\dots,Y_n;Z}} "3";"4"}
  \end{xy}
  \]
  in \(\mcC\) is commutative. Similarly, for each morphism \(g:Y\to
  Z\) in \(\mcC\), there exists a unique morphism
  \(\und\mcC(X_1,\dots,X_n;g):\und\mcC(X_1,\dots,X_n;Y)\to\und\mcC(X_1,\dots,X_n;Z)\)
  such that the diagram
  \[
  \begin{xy}
    (-45,9)*+{X_1,\dots,X_n,\und\mcC(X_1,\dots,X_n;Y)}="1";
    (45,9)*+{X_1,\dots,X_n,\und\mcC(X_1,\dots,X_n;Z)}="2";
    (-45,-9)*+{Y}="3"; 
    (45,-9)*+{Z}="4";
    {\ar@{->}^-{1^\mcC_{X_1},\dots,1^\mcC_{X_n},\und\mcC(X_1,\dots,X_n;g)}
      "1";"2"}; {\ar@{->}_-{\ev^\mcC_{X_1,\dots,X_n;Y}} "1";"3"};
    {\ar@{->}^-{\ev^\mcC_{X_1,\dots,X_n;Z}} "2";"4"}; {\ar@{->}^-{g}
      "3";"4"};
  \end{xy}
  \]
  in \(\mcC\) is commutative. In particular, if \(n=0\), then our
  conventions force \(\und\mcC(;g)=g\). If \(n=1\), then
  \(\und\mcC(X;g)=\langle\ev^\mcC_{X;Y}\cdot g\rangle\).
\end{notation}

\begin{lemma}\label{lem-aux-identities}
  Suppose that \(f_1:X^1_1,\dots,X^{k_1}_1\to
  Y_1\), \dots, \(f_n:X^1_n,\dots,X^{k_n}_n\to Y_n\),
  \(g:Y_1,\dots,Y_n\to Z\) are morphisms in a closed multicategory
  \(\mcC\).
  \begin{itemize}
  \item[(a)] If \(k_1=0\), i.e., \(f_1\) is a morphism \(()\to Y_1\),
    then \((f_1,\dots,f_n)\cdot g\) is equal to the composite
    \begin{equation*}
      X^1_2,\dots,X^{k_2}_2,\dots,X^1_n,\dots,X^{k_n}_n\rto{f_2,\dots,f_n}
      Y_2,\dots,Y_n\rto{\langle g\rangle}\und\mcC(Y_1;Z)\rto{\und\mcC(f_1;Z)}\und\mcC(;Z)=Z.
    \end{equation*}
  \item[(b)] If \(k_1=1\), i.e., \(f_1\) is a morphism \(X^1_1\to
    Y_1\), then \(\langle (f_1,\dots,f_n)\cdot g\rangle\) is equal to
    the composite
    \begin{equation*}
      X^1_2,\dots,X^{k_2}_2,\dots,X^1_n,\dots,X^{k_n}_n\rto{f_2,\dots,f_n}
      Y_2,\dots,Y_n\rto{\langle g\rangle}\und\mcC(Y_1;Z)\rto{\und\mcC(f_1;Z)}\und\mcC(X^1_1;Z).
    \end{equation*}
  \item[(c)] If \(k_1\ge1\), then \(\langle (f_1,\dots,f_n)\cdot
    g\rangle\) is equal to the composite
    \begin{align*}
      X^2_1,\dots,X^{k_1}_1,X^1_2,\dots,X^{k_2}_2,\dots,X^1_n,\dots,X^{k_n}_n
      &\rto{\langle f_1\rangle,f_2,\dots,f_n}
      \und\mcC(X^1_1;Y_1),Y_2,\dots,Y_n
      \\
      &\rto[\hphantom{\langle f_1\rangle,f_2,\dots,f_n}]{1,\langle
        g\rangle}\und\mcC(X^1_1;Y_1),\und\mcC(Y_1;Z)
      \\
      &\rto[\hphantom{\langle
        f_1\rangle,f_2,\dots,f_n}]{\mu_{\und\mcC}}\und\mcC(X^1_1;Z).
    \end{align*}
  \item[(d)] if \(n=1\), then \(\langle f_1\cdot g\rangle=\bigl[
    X^2_1,\dots,X^{k_1}_1\rto{\langle
      f_1\rangle}\und\mcC(X^1_1;Y_1)\rto{\und\mcC(X^1_1;g)}\und\mcC(X^1_1;Z)
    \bigr]\).
  \end{itemize}
\end{lemma}

\begin{proof}
  The proofs are easy and consist of checking the definitions. For
  example, in order to prove (a) note that
  \[
  \und\mcC(f_1;Z)=\bigl[
  \und\mcC(Y_1;Z)\rto{f_1,1^\mcC_{\und\mcC(Y_1;Z)}}Y_1,\und\mcC(Y_1;Z)\rto{\ev^\mcC_{Y_1;Z}}Z
  \bigr],
  \]
  therefore the composite in (a) is equal to
  \begin{multline*}
    \bigl[
    X^1_2,\dots,X^{k_2}_2,\dots,X^1_n,\dots,X^{k_n}_n\rto{f_2,\dots,f_n}
    Y_2,\dots,Y_n\rto{\langle
      g\rangle}\und\mcC(Y_1;Z)\rto{f_1,1^\mcC_{\und\mcC(Y_1;Z)}}
    Y_1,\und\mcC(Y_1;Z)\rto{\ev^\mcC_{Y_1;Z}}Z \bigr]
    \\
    =\bigl[
    X^1_2,\dots,X^{k_2}_2,\dots,X^1_n,\dots,X^{k_n}_n\rto{f_1,f_2,\dots,f_n}
    Y_1,Y_2,\dots,Y_n\rto{1^\mcC_{Y_1},\langle
      g\rangle}Y_1,\und\mcC(Y_1;Z)\rto{\ev^\mcC_{Y_1;Z}}Z \bigr].
  \end{multline*}
  The last two arrows compose to
  \(\varphi^\mcC_{Y_1;Y_2,\dots,Y_n;Z}(\langle
  g\rangle)=g:Y_1,\dots,Y_n\to Z\), hence the whole composite is equal
  to \((f_1,\dots,f_n)\cdot g\).
\end{proof}

\begin{lemma}\label{lem-C1fg-C1f-C1g}
  Let \(f:X\to Y\) and \(g:Y\to Z\) be morphisms in a closed
  multicategory \(\mcC\). Then for each \(W\in\Ob\mcC\) holds
  \(\und\mcC(W;f\cdot g)=\und\mcC(W;f)\cdot\und\mcC(W;g)\).
\end{lemma}

\begin{proof}
  The composite \(\und\mcC(W;f)\cdot\und\mcC(W;g)\) can be written as
  \[
  \und\mcC(W;X)\rto{\langle\ev^\mcC_{W;X}\cdot
    f\rangle}\und\mcC(W;Y)\rto{\und\mcC(W;g)}\und\mcC(W;Z),
  \]
  which is equal to \(\langle\ev^\mcC_{W;X}\cdot f\cdot
  g\rangle=\und\mcC(W;f\cdot g)\) by
  \propref{lem-aux-identities},~(d).
\end{proof}

\begin{lemma}
  Let \(f:W\to X\) and \(g:X\to Y\) be morphisms in a closed
  multicategory \(\mcC\). Then for each \(Z\in\Ob\mcC\) holds
  \(\und\mcC(f\cdot g;Z)=\und\mcC(g;Z)\cdot\und\mcC(f;Z)\).
\end{lemma}

\begin{proof}
  The composite \(\und\mcC(g;Z)\cdot\und\mcC(f;Z)\) can be written as
  \[
  \und\mcC(Y;Z)
  \rto{\langle(g,1^\mcC_{\und\mcC(Y;Z)})\cdot\ev^\mcC_{Y;Z}\rangle}
  \und\mcC(X;Z) \rto{\und\mcC(f;Z)}\und\mcC(W;Z),
  \]
  which is equal to \(\langle(f,1^\mcC_{\und\mcC(Y;Z)})\cdot
  ((g,1^\mcC_{\und\mcC(Y;Z)})\cdot\ev^\mcC_{Y;Z})\rangle=\langle(f\cdot
  g,1^\mcC_{\und\mcC(Y;Z)})\cdot\ev^\mcC_{Y;Z}\rangle=\und\mcC(f\cdot
  g;Z)\) by \propref{lem-aux-identities},~(b).
\end{proof}

\begin{lemma}\label{lem-Cf1-C1g-C1g-Cf1}
  Let \(f:W\to X\) and \(g:Y\to Z\) be morphisms in a closed
  multicategory \(\mcC\). Then
  \(\und\mcC(f;Y)\cdot\und\mcC(W;g)=\und\mcC(X;g)\cdot\und\mcC(f;Z)\).
\end{lemma}

\begin{proof}
  Both sides of the equation are equal to
  \(\langle(f,1^\mcC_{\und\mcC(X;Y)})\cdot\ev^\mcC_{X;Y}\cdot
  g\rangle\) by \propref{lem-aux-identities},~(b),(d).
\end{proof}

Lemmas~\ref{lem-C1fg-C1f-C1g}--\ref{lem-Cf1-C1g-C1g-Cf1} imply that
there exists a functor \(\und\mcC(-,-):\mcC^\op\times\mcC\to\mcC\),
\((X,Y)\mapsto\und\mcC(X;Y)\), defined by the formula
\(\und\mcC(f;g)=\und\mcC(f;Y)\cdot\mcC(W;g)=\und\mcC(X;g)\cdot\und\mcC(f;Z)\)
for each pair of morphisms \(f:W\to X\) and \(g:Y\to Z\) in \(\mcC\).

For each \(X,Y,Z\in\Ob\mcC\) there is a morphism
\(L^X_{YZ}:\und\mcC(Y;Z)\to\und\mcC(\und\mcC(X;Y);\und\mcC(X;Z))\)
uniquely determined by the equation
\begin{equation}
  \bigl[
  \und\mcC(X;Y),\und\mcC(Y;Z)\rto{1,L^X_{YZ}}\und\mcC(X;Y),
  \und\mcC(\und\mcC(X;Y);\und\mcC(X;Z))
  \rto{\ev^\mcC}\und\mcC(X;Z) \bigr]=\mu_{\und\mcC}.
  \label{equ-LXYZ-def}
\end{equation}

\begin{proposition}\label{prop-LX-C-functor}
  There is a \(\mcC\)\n-functor \(L^X:\und\mcC\to\und\mcC\),
  \(Y\mapsto\und\mcC(X;Y)\), with the action on \(\Hom\)\n-objects
  given by
  \(L^X_{YZ}:\und\mcC(Y;Z)\to\und\mcC(\und\mcC(X;Y);\und\mcC(X;Z))\).
\end{proposition}

\begin{proof}
  That so defined \(L^X\) preserves identities is a consequence of the
  identity axiom. The compatibility with composition is established as
  follows. Consider the diagram
  \[
  \begin{xy}
    (-60,12)*+{%
      \begin{array}{l}
        \und\mcC(X;Y),\\
        \und\mcC(Y;Z),\\
        \und\mcC(Z;W)
      \end{array}
    }="1"; (0,12)*+{%
      \begin{array}{l}
        \und\mcC(X;Y),\\
        \und\mcC(\und\mcC(X;Y);\und\mcC(X;Z)),\\
        \und\mcC(\und\mcC(X;Z);\und\mcC(X;W))
      \end{array}
    }="2"; (60,12)*+{%
      \begin{array}{l}
        \und\mcC(X;Z),\\
        \und\mcC(\und\mcC(X;Z);\und\mcC(X;W))
      \end{array}
    }="3"; (-60,-12)*+{%
      \begin{array}{l}
        \und\mcC(X;Y),\\
        \und\mcC(Y;W)
      \end{array}
    }="4"; (0,-12)*+{%
      \begin{array}{l}
        \und\mcC(X;Y),\\
        \und\mcC(\und\mcC(X;Y);\und\mcC(X;W))
      \end{array}
    }="5"; (60,-12)*+{\und\mcC(X;W)}="6";
    {\ar@{->}^-{1,L^X_{YZ},L^X_{ZW}} "1";"2"}; {\ar@{->}^-{\ev^\mcC,1}
      "2";"3"}; {\ar@{->}_-{1,\mu_{\und\mcC}} "1";"4"};
    {\ar@{->}^-{1,\mu_{\und\mcC}} "2";"5"}; {\ar@{->}^-{\ev^\mcC}
      "3";"6"}; {\ar@{->}^-{1,L^X_{YW}} "4";"5"};
    {\ar@{->}^-{\ev^\mcC} "5";"6"};
  \end{xy}
  \]
  By the definition of \(L^X\) the exterior expresses the
  associativity of \(\mu_{\und\mcC}\). The right square is the
  definition of \(\mu_{\und\mcC}\). By the closedness of \(\mcC\) the
  square
  \[
  \begin{xy}
    (-40,9)*+{\und\mcC(Y;Z),\und\mcC(Z;W)}="1";
    (40,9)*+{\und\mcC(\und\mcC(X;Y);\und\mcC(X;Z)),
      \und\mcC(\und\mcC(X;Z);\und\mcC(X;W))}="2";
    (-40,-9)*+{\und\mcC(Y;W)}="3";
    (40,-9)*+{\und\mcC(\und\mcC(X;Y);\und\mcC(X;W))}="4";
    {\ar@{->}^-{L^X_{YZ},L^X_{ZW}} "1";"2"};
    {\ar@{->}_-{\mu_{\und\mcC}} "1";"3"}; {\ar@{->}^-{\mu_{\und\mcC}}
      "2";"4"}; {\ar@{->}^-{L^X_{YW}} "3";"4"};
  \end{xy}
  \]
  is commutative, hence the assertion.
\end{proof}

\begin{definition}[{\cite[Section~4.18]{BLM}}]
  Let $\mcC$, $\mcD$ be multicategories. Let \(F:\mcC\to\mcD\) be a
  multifunctor. For each \(X_1,\dots,X_m,Z\in\Ob\mcC\), define a
  morphism in $\mcD$
  \begin{equation*}
    \und{F}_{X_1,\dots,X_m;Z}: F\und\mcC(X_1,\dots,X_m;Z)
    \to \und\mcD(FX_1,\dots,FX_m;FZ)
  \end{equation*}
  as the only morphism that makes the diagram
  \[
    \begin{xy}
      (-25,0)*+{FX_1,\dots,FX_m,F\und\mcC(X_1,\dots,X_m;Z)}="fxfc";
      (25,15)*+{FX_1\dots,FX_m,\und\mcD(FX_1,\dots,FX_m;FZ)}="fxd";
      (25,-15)*+{FZ}="fy";
      {\ar@{->}^(.4){1^\mcD_{FX_1},\dots,1^\mcD_{FX_m},\und{F}_{X_1,\dots,X_m;Z}\hskip3em} "fxfc";"fxd"};
      {\ar@{->}_(.4){F\ev^\mcC_{X_1,\dots,X_m;Z}} "fxfc";"fy"};
      {\ar@{->}^-{\ev^\mcD_{FX_1,\dots,FX_m;FZ}} "fxd";"fy"};
    \end{xy}
  \]
  commute. It is called the \emph{closing transformation} of the
  multifunctor $F$.
\end{definition}

The following properties of closing transformations can be found in
\cite[Section~4.18]{BLM}. To keep the exposition self-contained we
include their proofs here.

\begin{proposition}[{\cite[Lemma~4.19]{BLM}}]
  \label{prop-F-D1undF-phi-phi-F}
  The diagram
  \begin{equation}
    \begin{xy}
      (-40,18)*+{\mcC\bigl(Y_1,\dots,Y_n;\und\mcC(X_1,\dots,X_m;Z)\bigr)}="1";
      (40,18)*+{\mcD\bigl(FY_1,\dots,FY_n;F\und\mcC(X_1,\dots,X_m;Z)\bigr)}="2";
      (40,0)*+{\mcD\bigl(FY_1,\dots,FY_n;\und\mcD(FX_1,\dots,FX_m;FZ)\bigr)}="3";
      (-40,-18)*+{\mcC\bigl(X_1,\dots,X_m,Y_1,\dots,Y_n;Z\bigr)}="4";
      (40,-18)*+{\mcD\bigl(FX_1,\dots,FX_m,FY_1,\dots, FY_n;FZ\bigr)}="5";
      {\ar@{->}^-{F} "1";"2"};
      {\ar@{->}^-{\mcD(1;\und{F}_{X_1,\dots,X_m;Z})} "2";"3"};
      {\ar@{->}^-{\varphi^\mcD_{FX_1,\dots,FX_m;FY_1,\dots,FY_n;FZ}} "3";"5"};
      {\ar@{->}_-{\varphi^\mcC_{X_1,\dots,X_m;Y_1,\dots,Y_n;Z}} "1";"4"};
      {\ar@{->}^-{F} "4";"5"};
    \end{xy}
    \label{equ-F-underline-phi}
  \end{equation}
  commutes, for each \(m,n\in\NN\) and objects
  \(X_i,Y_j,Z\in\Ob\mcC\), \(1\le i\le m\), \(1\le j\le n\).
\end{proposition}

\begin{proof}
  Pushing an arbitrary morphism
  \(g:Y_1,\dots,Y_n\to\und\mcC(X_1,\dots,X_m;Z)\) along the top-right
  path produces the composite
  \begin{align*}
    FX_1,\dots,FX_m,FY_1,\dots,FY_n
    &\rto[\hphantom{1^\mcD_{FX_1}, \dots, 1^\mcD_{FX_m},
      \und{F}_{(X_i);Z}}]{1^\mcD_{FX_1},\dots,1^\mcD_{FX_m},Fg}
    FX_1,\dots,FX_m,F\und\mcC(X_1,\dots,X_m;Z)
    \\
    &\rto{1^\mcD_{FX_1},\dots,1^\mcD_{FX_m},\und{F}_{(X_i);Z}}
    FX_1,\dots,FX_m,\mcD(FX_1,\dots,FX_m;FZ)
    \\
    &\rto[\hphantom{1^\mcD_{FX_1}, \dots, 1^\mcD_{FX_m},
      \und{F}_{(X_i);Z}}]{\ev^\mcD_{FX_1,\dots,FX_m;FZ}}FZ.
  \end{align*}
  The composition of the last two arrows is equal to
  \(F\ev^\mcC_{X_1,\dots,X_m;Z}\) by the definition of
  \(\und{F}_{X_1,\dots,X_m;Z}\). Since \(F\) preserves composition and
  identities, the above composite equals
  \[
  F\bigl((1^\mcC_{X_1},\dots,1^\mcC_{X_m},g)\cdot
  \ev^\mcC_{X_1,\dots,X_m;Z}\bigr)
  =F\bigl(\varphi_{X_1,\dots,X_m;Y_1,\dots,Y_n;Z}(g)\bigr),
  \]
  hence the assertion.
\end{proof}

Let \(F:\mcV\to\mcW\) be a multifunctor, and let \(\cc\) be a
\(\mcV\)\n-category. Then we obtain a \(\mcW\)\n-category \(F_*\cc\)
with the same set of objects if we define its \(\Hom\)\n-objects by
\((F_*\cc)(X,Y)=F\cc(X,Y)\), and composition and identities by
respectively \(\mu_{F_*\cc}=F(\mu_\cc):F\cc(X,Y),F\cc(Y,Z)\to
F\cc(X,Z)\) and \(1^{F_*\cc}_X=F(1^\cc_X):()\to F\cc(X,X)\).

\begin{proposition}[cf.~{\cite[Proposition~4.21]{BLM}}]
  \label{prop-und-F-D-functor}
  Let \(F:\mcC\to\mcD\) be a multifunctor between closed
  multicategories. There is \(\mcD\)\n-functor
  \(\und{F}:F_*\und\mcC\to\und\mcD\), \(X\mapsto FX\), such that
  \[
  \und{F}_{X;Y}:(F_*\und\mcC)(X;Y)=F\und\mcC(X;Y)\to\und\mcD(FX;FY)
  \]
  is the closing transformation, for each \(X,Y\in\Ob\mcC\).
\end{proposition}

\begin{proof}
  First, let us check that \(\und{F}\) preserves identities. In other
  words, we must prove the equation
  \[
  \bigl[
  ()\rto{F1^{\und\mcC}_X}F\und\mcC(X;X)\rto{\und{F}_{X,X}}\und\mcD(FX;FX)
  \bigr]=1^{\und\mcD}_{FX}.
  \]
  Let us check that the left hand side solves the equation that
  determines the right hand side. We have:
  \begin{multline*}
    \bigl[ FX\rto{1^\mcC_{FX},F1^{\und\mcC}_X}FX,F\und\mcC(X;X)
    \rto{1^\mcC_{FX},\und{F}_{X,X}}FX, \und\mcD(FX;FX)\rto{\ev^\mcD}FX
    \bigr]
    \\
    =\bigl[
    FX\rto{1^\mcC_{FX},F1^{\und\mcC}_X}FX,F\und\mcC(X;X)\rto{F\ev^\mcC}FX
    \bigr]=F[(1^\mcC_{FX},1^{\und\mcC}_X)\cdot\ev^\mcC]=F1^\mcC_X=1^\mcD_{FX}.
  \end{multline*}

  To show that \(\und F\) preserves composition, we must show that the
  diagram
  \begin{equation}
    \begin{xy}
      (-25,9)*+{F\und\mcC(X;X),F\und\mcC(Y;Z)}="ff";
      (25,9)*+{F\und\mcC(X;Z)}="f";
      (-25,-9)*+{\und\mcD(FX;FY),\und\mcD(FY;FZ)}="dd";
      (25,-9)*+{\und\mcD(FX;FZ)}="d"; {\ar@{->}^-{F\mu_{\und\mcC}}
        "ff";"f"}; {\ar@{->}^-{\mu_{\und\mcD}} "dd";"d"};
      {\ar@{->}_-{\und{F}_{X,Y},\und{F}_{Y,Z}} "ff";"dd"};
      {\ar@{->}^-{\und{F}_{X,Z}} "f";"d"};
    \end{xy}
    \label{equ-underline-F-mu-mu-underline-F}
  \end{equation}
  commutes. This follows from \diaref{dia-Fmu-undF-undF-undF-mu}. The
  lower diamond is the definition of \(\mu_{\und\mcD}\). The exterior
  commutes by the definition of \(\mu_{\und\mcC}\) and because \(F\)
  preserves composition. The left upper diamond and both triangles
  commute by the definition of the closing transformation.
\end{proof}

\begin{figure}
  \[
  \begin{xy}
    (0,36)*+{FX,F\und\mcC(X;Y),F\und\mcC(Y;Z)}="1";
    (-40,18)*+{FY,F\und\mcC(Y;Z)}="2";
    (40,18)*+{FX,F\und\mcC(X;Z)}="3";
    (0,0)*+{FX,\und\mcD(FX;FY),\und\mcD(FY;FZ)}="4";
    (-40,-18)*+{FY,\und\mcD(FY;FZ)}="5";
    (40,-18)*+{FX,\und\mcD(FX;FZ)}="6"; (0,-36)*+{FZ}="7";
    {\ar@{->}_-{F\ev^\mcC_{X;Y},1} "1";"2"};
    {\ar@{->}^-{F\mu_{\und\mcC}} "1";"3"};
    {\ar@{->}|-{1,\und{F}_{X;Y},\und{F}_{Y;Z}\vphantom{\Bigl|}}
      "1";"4"}; {\ar@{->}^-{1,\und{F}_{Y;Z}} "2";"5"};
    {\ar@{->}^(.5){\hskip1.5em\ev^\mcD_{FX;FY},1} "4";"5"};
    {\ar@{->}_(.5){1,\mu_{\und\mcD}} "4";"6"};
    {\ar@{->}_-{1,\und{F}_{X,Z}} "3";"6"};
    {\ar@{->}^-{\ev^\mcD_{FY;FZ}} "5";"7"};
    {\ar@{->}_-{\ev^\mcD_{FX;FZ}} "6";"7"};
    {\ar@{->}@/_{10pc}/^(.6){F\ev^\mcC_{Y;Z}} "2";"7"};
    {\ar@{->}@/^{10pc}/_(.6){F\ev^\mcC_{X;Z}} "3";"7"}; (0,-55)*{};
  \end{xy}
  \]
  \caption{\label{dia-Fmu-undF-undF-undF-mu}}
\end{figure}

\begin{lemma}[{\cite[Lemma~4.25]{BLM}}]
  \label{lem-underlineGF-GunderlineF-underlineG}
  Let $\mcC$, $\mcD$, $\mcE$ be closed multicategories, and let \(\mcC
  \rto{F} \mcD \rto{G} \mcE\) be multifunctors. Then
  \begin{align*}
    \und{G\circ F}_{X_1,\dots,X_m;Y} = \bigl[
    GF\und\mcC(X_1,\dots,X_m;Y)
    &\rto[\hphantom{\und{G}_{FX_1,\dots,FX_m;FY}}]{G\und{F}_{X_1,\dots,X_m;Y}}
    G\und\mcD(FX_1,\dots,FX_m;FY)
    \\
    &\rto{\und{G}_{FX_1,\dots,FX_m;FY}}
    \und{\mcE}(GFX_1,\dots,GFX_m;GFY) \bigr].
  \end{align*}
\end{lemma}

\begin{proof}
  This follows from the commutative diagram
  \[
  \begin{xy}
    (-45,9)*+{GFX_1,\dots,GFX_m,G\und\mcD(FX_1,\dots,FX_m;FY)}="2";
    (45,9)*+{GFX_1,\dots,GFX_m,\und{\mcE}(GFX_1,\dots,GFX_m;GFY)}="3";
    (-45,-9)*+{GFX_1,\dots,GFX_m,GF\und\mcC(X_1,\dots,X_m;Y)}="1";
    (45,-9)*+{GFY}="4";
    {\ar@{->}^-{1^\mcE_{GFX_1},\dots,1^\mcE_{GFX_m},G\und{F}_{X_1,\dots,X_m;Y}} "1";"2"};
    {\ar@{->}@/^{2pc}/^-{1^\mcE_{GFX_1},\dots,1^\mcE_{GFX_m},\und{G}_{FX_1,\dots,FX_m;FY}} "2";"3"};
    {\ar@{->}^-{\ev^\mcE_{GFX_1,\dots,GFX_m;GFY}} "3";"4"};
    {\ar@{->}^(.6){GF\ev^\mcC_{X_1,\dots,X_m;Y}} "1";"4"};
    {\ar@{->}^-{G\ev^\mcD_{FX_1,\dots,FX_m;FY}} "2";"4"};
  \end{xy}
  \]
  The upper triangle is the definition of
  \(\und{G}_{FX_1,\dots,FX_m;FY}\), the lower triangle commutes by the
  definition of \(\und{F}_{X_1,\dots,X_m;Y}\) and because \(G\)
  preserves composition.
\end{proof}

\begin{proposition}[{\cite[Lemma~4.24]{BLM}}]
  \label{prop-multinatural-transformation-nuFGCD}
  Let $\nu:F\to G:\mcC\to\mcD$ be a multinatural transformation of
  multifunctors between closed multicategories. Then the diagram
  \begin{equation}
    \begin{xy}
      (-59,9)*+{F\und\mcC(X_1,\dots,X_m;Y)}="1";
      (71,9)*+{\und\mcD(FX_1,\dots,FX_m;FY)}="2";
      (-59,-9)*+{G\und\mcC(X_1,\dots,X_m;Y)}="3";
      (0,-9)*+{\und\mcD(GX_1,\dots,GX_m;GY)}="4";
      (71,-9)*+{\und\mcD(FX_1,\dots,FX_m;GY)}="5";
      {\ar@{->}^-{\und{F}_{X_1,\dots,X_m;Y}} "1";"2"};
      {\ar@{->}_-{\nu_{\und\mcC(X_1,\dots,X_m;Y)}} "1";"3"};
      {\ar@{->}^-{\und\mcD(FX_1,\dots,FX_m;\nu_Y)} "2";"5"};
      {\ar@{->}^-{\und{G}_{X_1,\dots,X_m;Y}} "3";"4"};
      {\ar@{->}^-{\und\mcD(\nu_{X_1},\dots,\nu_{X_m};GY)} "4";"5"};
    \end{xy}
    \label{eq-cl-nt}
  \end{equation}
  is commutative.
\end{proposition}

\begin{figure}
  \begin{center}
    %\resizebox{!}{.9\texthigh}{
    \rotatebox{90}{%
      \begin{xy}
        (-100,60)*+{%
          \begin{array}{c}
            FX_1,\dots,FX_m,\\[1pt]
            F\und\mcC(X_1,\dots,X_m;Y)
          \end{array}
        }="1";
        (-100,30)*+{%
          \begin{array}{c}
            FX_1,\dots,FX_m,\\[1pt]
            G\und\mcC(X_1,\dots,X_m;Y)
          \end{array}
        }="2";
        (50,30)*+{%
          \begin{array}{c}
            FX_1,\dots,FX_m,\\[1pt]
            \und\mcD(FX_1,\dots,FX_m;FY)
          \end{array}
        }="3";
        (-50,0)*+{%
          \begin{array}{c}
            FX_1,\dots,FX_m,\\[1pt]
            \und\mcD(GX_1,\dots,GX_m;GY)
          \end{array}
        }="4";
        (50,0)*+{%
          \begin{array}{c}
            FX_1,\dots,FX_m,\\[1pt]
            \und\mcD(FX_1,\dots,FX_m;GY)
          \end{array}
        }="5";
        (-100,-60)*+{%
          \begin{array}{c}
            GX_1,\dots,GX_m,\\[1pt]
            G\und\mcC(X_1,\dots,X_m;Y)
          \end{array}
        }="6";
        (100,60)*+{FY}="7";
        (-50,-30)*+{%
          \begin{array}{c}
            GX_1,\dots,GX_m,\\[1pt]
            \und\mcD(GX_1,\dots,GX_m;GY)
          \end{array}
        }="8";
        (100,-60)*+{GY}="9";
        {\ar@{->}^-{1^\mcD_{FX_1},\dots,1^\mcD_{FX_m},\nu_{\und\mcC(X_1,\dots,X_m;Y)}\hskip1em} "1";"2"};
        {\ar@{->}^-{\hskip1em 1^\mcD_{FX_1},\dots,1^\mcD_{FX_m},\und{F}} "1";"3"};
        {\ar@{->}^-{\hskip1em 1^\mcD_{FX_1},\dots,1^\mcD_{FX_m},\und{G}} "2";"4"};
        {\ar@{->}^-{1^\mcD_{FX_1},\dots,1^\mcD_{FX_m},\und\mcD(\nu_{X_1},\dots,\nu_{X_m};GY)} "4";"5"};
        {\ar@{->}_-{1^\mcD_{FX_1},\dots,1^\mcD_{FX_m},\und\mcD(FX_1,\dots,FX_m;\nu_Y)} "3";"5"};
        {\ar@{->}^-{\nu_{X_1},\dots,\nu_{X_m},1^\mcD_{G\und\mcC(X_1,\dots,X_m;Y)}} "2";"6"};
        {\ar@{->}^-{\ev^\mcD} "3";"7"};
        {\ar@{->}_-{\hskip1em 1^\mcD_{GX_1},\dots,1^\mcD_{GX_m},\und{G}} "6";"8"};
        {\ar@{->}_-{\nu_Y} "7";"9"};
        {\ar@{->}^-{\nu_{X_1},\dots,\nu_{X_m},1^\mcD_{\und\mcD(GX_1,\dots,GX_m;GY)}} "4";"8"};
        {\ar@{->}_-{\ev^\mcD} "5";"9"};
        {\ar@{->}^-{\ev^\mcD} "8";"9"};
        {\ar@{->}^-{F\ev^\mcC} "1";"7"};
        {\ar@{->}_-{G\ev^\mcC} "6";"9"};
      \end{xy}
    }
    % }
  \end{center}
  \caption{\label{dia-undF-D1nu-nu-undG-Dnu1}}
\end{figure}

\begin{proof}
  The claim follows from \diaref{dia-undF-D1nu-nu-undG-Dnu1}. Its
  exterior commutes by the multinaturality of \(\nu\). The
  quadrilateral in the middle is the definition of
  \(\und\mcD(\nu_{X_1},\dots,\nu_{X_m};GY)\). The trapezoid on the
  right is the definition of \(\und\mcD(FX_1,\dots,FX_m;\nu_Y)\). The
  triangles commute by the definition of closing transformation.
\end{proof}

\section{From closed multicategories to closed categories}
\label{sec:from-clos-mult}

A closed category comes equipped with a distinguished object
\(\unito\). We want to produce a closed category out of a closed
multicategory, so we need a notion of a closed multicategory with a
unit object. We introduce it in somewhat ad hoc fashion, which is
sufficient for our purposes though. Similarly to closedness,
possession of a unit object is a property of a closed multicategory
rather than additional data.

\begin{definition}
  Let \(\mcC\) be a closed multicategory. A \emph{unit object} of
  \(\mcC\) is an object \(\unito\in\Ob\mcC\) together with a morphism
  \(u:()\to\unito\) such that,  for each \(X\in\Ob\mcC\), the morphism
  \[
  \und\mcC(u;1):\und\mcC(\unito;X)\to\und\mcC(;X)=X
  \]
  is an isomorphism.
\end{definition}

\begin{remark}\label{rem-C1X-C0X-iso}
  If \(\unito\) is a unit object of a closed multicategory \(\mcC\),
  then \(\mcC(u;X):\mcC(\unito;X)\to\mcC(;X)\) is a bijection. This
  follows from the equation
  \[
  \bigl[\mcC(;\und\mcC(\unito;X)) \rto[\sim]{\varphi^\mcC}
  \mcC(\unito;X) \rto{\mcC(u;X)} \mcC(;X) \bigr] =
  \mcC(;\und\mcC(u;X)),
  \]
  which is an immediate consequence of the definitions. The
  bijectivity of \(\mcC(u;X)\) can be stated as the following
  universal property: for each morphism \(f:()\to X\), there exists a
  unique morphism \(\overline{f}:\unito\to X\) such that
  \(u\cdot\overline{f}=f\). In particular, a unit object, if it
  exists, is unique up to isomorphism.
\end{remark}

\begin{proposition}\label{prop-cl-multicat-cl-cat}
  A closed multicategory \(\mcC\) with a unit object gives rise to a
  closed category
  \((\cc,\linebreak[1]\und\cc(-,-)\linebreak[1],\unito,i,j,L)\),
  where:
  \begin{itemize}
  \item \(\cc\) is the underlying category of the multicategory \(\mcC\);
  \item \(\und\cc(X,Y)=\und\mcC(X;Y)\), for each \(X,Y\in\Ob\mcC\);
  \item \(\unito\) is the unit object of \(\mcC\);
  \item
    \(i_X=\bigl(\und\mcC(u;X)\bigr)^{-1}:X=\und\mcC(;X)\to\und\mcC(\unito;X)\);
  \item \(j_X=\overline{1^{\und\mcC}_X}:\unito\to\und\mcC(X;X)\) is a
    unique morphism such that
    \(\bigl[()\rto{u}\unito\rto{j_X}\und\mcC(X;X)\bigr]=1^{\und\mcC}_X\);
  \item
    \(L^X_{YZ}:\und\mcC(Y;Z)\to\und\mcC(\und\mcC(X;Y);\und\mcC(X;Z))\)
    is determined uniquely by equation \eqref{equ-LXYZ-def}.
  \end{itemize}
\end{proposition}

We shall call \(\cc\) the \emph{underlying closed category} of
\(\mcC\). Usually we do not distinguish notationally between a closed
multicategory and its underlying closed category; this should lead to
minimal confusion.

\begin{proof}
  We leave it as an easy exercise for the reader to show the naturality of
  \(i_X\), \(j_X\), and \(L^X_{YZ}\), and proceed directly to checking
  the axioms.

  \medskip

  \noindent\textbf{CC1.} By \remref{rem-C1X-C0X-iso} the equation
  \[
  \bigl[
  \unito\rto{j_Y}\und\mcC(Y;Y)\rto{L^X_{YY}}\und\mcC(\und\mcC(X;Y);\und\mcC(X;Y))
  \bigr]=j_{\und\mcC(X;Y)}
  \]
  is equivalent to the equation
  \[
  \bigl[
  ()\rto[1^{\und\mcC}_Y]{u\cdot
    j_Y}\und\mcC(Y,Y)\rto{L^X_{YY}}\und\mcC(\und\mcC(X;Y);\und\mcC(X;Y))
  \bigr]=u\cdot j_{\und\mcC(X;Y)}=1^{\und\mcC}_{\und\mcC(X;Y)},
  \]
  which expresses the fact that the \(\mcC\)\n-functor \(L^X\)
  preserves identities.

  \medskip

  \noindent\textbf{CC2.} The equation in question
  \[
  \bigl[
  \und\mcC(X;Y)\rto{L^X_{XY}}\und\mcC(\und\mcC(X;X);\und\mcC(X;Y))
  \rto{\und\mcC(j_X;1)}\und\mcC(\unito;\und\mcC(X;Y))
  \bigr]=i_{\und\mcC(X;Y)}=(\und\mcC(u;1))^{-1}
  \]
  is equivalent to
  \[
  \bigl[
  \und\mcC(X;Y)\rto{L^X_{XY}}\und\mcC(\und\mcC(X;X);\und\mcC(X;Y))
  \rto[\und\mcC(1^{\und\mcC}_X;1)]{\und\mcC(u\cdot
    j_X;1)}\und\mcC(;\und\mcC(X;Y))=\und\mcC(X;Y)
  \bigr]=1^\mcC_{\und\mcC(X;Y)}.
  \]
  The left hand side is equal to
  \begin{multline*}
    \bigl[
    \und\mcC(X;Y)\rto{L^X_{YZ}}\und\mcC(\und\mcC(X;X);\und\mcC(X;Y))\rto{1^{\und\mcC
      } _X , 1 }
    \und\mcC(X;X),\und\mcC(\und\mcC(X;X);\und\mcC(X;Y))
    \rto{\ev^\mcC}\und\mcC(X;Y)
    \bigr]
    \quad
    \\
    \qquad
    =\bigl[
    \und\mcC(X;Y)\rto{1^{\und\mcC}_X,1}\und\mcC(X;X),\und\mcC(X;Y)\rto{1,L^X_{YZ}}
    \und\mcC(X;X),\und\mcC(\und\mcC(X;X);\und\mcC(X;Y))\rto{\ev^\mcC}\und\mcC(X;Y)
    \bigr]
    \hfill
    \\
    \qquad
    =\bigl[
    \und\mcC(X;Y)\rto{1^{\und\mcC}_X,1}\und\mcC(X;X),\und\mcC(X;Y)\rto{\mu_{\und\mcC}}\und\mcC(X;Y)
    \bigr]=1^\mcC_{\und\mcC(X;Y)}
    \hfill
  \end{multline*}
  by the identity axiom in the \(\mcC\)\n-category \(\und\mcC\).

  \medskip

  \noindent\textbf{CC3.} The commutativity of the diagram
  \[
  \begin{xy}
    (-45,18)*+{\und\mcC(U;V)}="1";
    (45,18)*+{\und\mcC(\und\mcC(Y;U);\und\mcC(Y;V))}="2";
    (-45,0)*+{\und\mcC(\und\mcC(X;U);\und\mcC(X;V))}="3";
    (45,-18)*+{\und\mcC(\und\mcC(Y;U);\und\mcC(\und\mcC(X;Y);\und\mcC(X;V)))}="4";
    (-45,-18)*+{\und\mcC(\und\mcC(\und\mcC(X;Y);\und\mcC(X;U));\und\mcC(\und\mcC(X;Y);\und\mcC(X;V)))}="5";
    {\ar@{->}^-{L^Y_{UV}} "1";"2"};
    {\ar@{->}_-{L^X_{UV}} "1";"3"};
    {\ar@{->}^-{\und\mcC(1;L^X_{YV})} "2";"4"};
    {\ar@{->}_-{L^{\und\mcC(X;Y)}_{\und\mcC(X;U),\und\mcC(X;V)}} "3";"5"};
    {\ar@{->}^-{\und\mcC(L^X_{YU};1)} "5";"4"};
  \end{xy}
  \]
  is equivalent by closedness to the commutativity of the exterior of
  \diaref{dia-mu-LX-LX-LX-mu}, which just expresses the fact that the
  \(\mcC\)\n-functor \(L^X:\und\mcC\to\und\mcC\) preserves
  composition and which is part of the assertion of \propref{prop-LX-C-functor}.
  \begin{figure}
    \begin{center}
      \rotatebox{90}{%
        \begin{xy}
          (-80,40)*+{\und\mcC(Y;U),\und\mcC(U;V)}="1";
          (80,40)*+{\und\mcC(Y;V)}="2";
          (35,20)*+{\und\mcC(Y;U),\und\mcC(\und\mcC(Y;U);\und\mcC(Y;V))}="3";
          (-45,0)*+{\und\mcC(Y;U),\und\mcC(\und\mcC(X;U);\und\mcC(X;V))}="4";
          (35,0)*+{\und\mcC(Y;U),\und\mcC(\und\mcC(Y;U);\und\mcC(\und\mcC(X;Y);\und\mcC(X;V)))}="5";
          (-5,-20)*+{\und\mcC(Y;U),\und\mcC(\und\mcC(\und\mcC(X;Y);\und\mcC(X;U));\und\mcC(\und\mcC(X;Y);\und\mcC(X;V)))}="6";
          (-5,-40)*+{\und\mcC(\und\mcC(X;Y);\und\mcC(X;U)),\und\mcC(\und\mcC(\und\mcC(X;Y);\und\mcC(X;U));\und\mcC(\und\mcC(X;Y);\und\mcC(X;V)))}="7";
          (-80,-60)*+{\und\mcC(\und\mcC(X;Y);\und\mcC(X;U)),\und\mcC(\und\mcC(X;U);\und\mcC(X;V))}="8";
          (80,-60)*+{\und\mcC(\und\mcC(X;Y);\und\mcC(X;V))}="9";
          {\ar@{->}^-{\mu_{\und\mcC}} "1";"2"};
          {\ar@{->}^-{1,L^Y_{UV}} "1";"3"};
          {\ar@{->}^-{1,L^X_{UV}} "1";"4"};
          {\ar@{->}_-{L^X_{YU},L^X_{UV}} "1";"8"};
          {\ar@{->}^-{L^X_{YV}} "2";"9"};
          {\ar@{->}^-{\ev^\mcC} "3";"2"};
          {\ar@{->}^-{1,\und\mcC(1;L^X_{YV})} "3";"5"};
          {\ar@{->}_(.4){1,L^{\und\mcC(X;Y)}_{\und\mcC(X;U),\und\mcC(X;V)}} "4";"6"};
          {\ar@{->}_-{L^X_{YU},1} "4";"8"};
          {\ar@{->}^-{\ev^\mcC} "5";"9"};
          {\ar@{->}_(.6){1,\und\mcC(L^X_{YU};1)} "6";"5"};
          {\ar@{->}^-{L^X_{YU},1} "6";"7"};
          {\ar@{->}^-{\ev^\mcC} "7";"9"};
          {\ar@{->}^-{1,L^{\und\mcC(X;Y)}_{\und\mcC(X;U),\und\mcC(X;V)}\hskip2em} "8";"7"};
          {\ar@{->}^-{\mu_{\und\mcC}} "8";"9"};
        \end{xy}
      }
    \end{center}
    \caption{\label{dia-mu-LX-LX-LX-mu}}
  \end{figure}

  \medskip

  \noindent\textbf{CC4.} The equation in question
  \[
  \bigl[
  \und\mcC(Y;Z)\rto{L^\unito}\und\mcC(\und\mcC(\unito;Y);\und\mcC(\unito;Z))
  \rto{\und\mcC(i_Y;1)}\und\mcC(Y;\und\mcC(\unito;Z))
  \bigr]=\und\mcC(1;i_Z)
  \]
  is equivalent to the equation
  \[
  \bigl[
  \und\mcC(Y;Z)\rto{L^\unito}\und\mcC(\und\mcC(\unito;Y);\und\mcC(\unito;Z))
  \rto{\und\mcC(1;\und\mcC(u;1))}\und\mcC(\und\mcC(\unito;Y);Z)
  \bigr]=\und\mcC(\und\mcC(u;1);1).
  \]
  The latter follows by closedness from the commutative diagram
  \[
  \begin{xy}
    (-50,27)*+{\und\mcC(\unito;Y),\und\mcC(Y;Z)}="1";
    (60,27)*+{\und\mcC(\unito;Y),\und\mcC(\und\mcC(\unito;Y),\und\mcC(\unito;Z))}="2";
    (20,9)*+{\und\mcC(\unito;Z)}="3";
    (60,9)*+{\und\mcC(\unito;Y),\und\mcC(\und\mcC(\unito;Y);Z)}="4";
    (-20,-9)*+{\unito,\und\mcC(\unito;Y),\und\mcC(Y;Z)}="5";
    (20,-9)*+{\unito,\und\mcC(\unito;Z)}="6";
    (-50,-27)*+{Y,\und\mcC(Y;Z)}="7";
    (60,-27)*+{Z}="8";
    {\ar@{->}^-{1,L^\unito} "1";"2"};
    {\ar@{->}^-{\mu_{\und\mcC}} "1";"3"};
    {\ar@{->}_-{\ev^{\mcC}} "2";"3"};
    {\ar@{->}^-{1,\und\mcC(1;\und\mcC(u;1))} "2";"4"};
    {\ar@{->}^-{u,1,1} "1";"5"};
    {\ar@{->}_-{\und\mcC(u;1),1} "1";"7"};
    {\ar@{->}^-{1,\mu_{\und\mcC}} "5";"6"};
    {\ar@{->}^-{\und\mcC(u;1)} "3";"8"};
    {\ar@{->}^-{\ev^{\mcC}} "4";"8"};
    {\ar@{->}_(.4){\ev^{\mcC}} "6";"8"};
    {\ar@{->}^-{u,1} "3";"6"};
    {\ar@{->}^(.4){\ev^{\mcC},1} "5";"7"};
    {\ar@{->}^-{\ev^{\mcC}} "7";"8"};
  \end{xy}
  \]
  in which the bottom quadrilateral is the definition of
  \(\mu_{\und\mcC}\), the right hand side quadrilateral is the definition
  of the morphism \(\und\mcC(1;\und\mcC(u;1))\), the top triangle is the
  definition of \(L^\unito\), and the remaining triangles commute by the
  definition of \(\und\mcC(u;1)\).

  \medskip

  \noindent\textbf{CC5.} A straightforward computation shows that the
  composite
  \[
  \mcC(X;Y)\rto{\gamma}\mcC(\unito;\und\mcC(X;Y))\rto[\sim]{\mcC(u;1)}\mcC(;\und\mcC(X;Y))
  \rto[\sim]{\varphi^\mcC}\mcC(X;Y)
  \]
  is the identity map, which readily implies that \(\gamma\) is a
  bijection.

  The proposition is proven.
\end{proof}

\begin{proposition}\label{prop-multifun-cl-fun}
  Let \(\mcC\) and \(\mcD\) be closed multicategories with unit objects.
  Let \(\cc\) and \(\cd\) denote the corresponding underlying closed
  categories. A multifunctor \(F:\mcC\to\mcD\) gives rise to a closed
  functor \(\Phi=(\phi,\hat\phi,\phi^0):\cc\to\cd\), where:
  \begin{itemize}
  \item \(\phi:\cc\to\cd\) is the underlying functor of the multifunctor
    \(F\);
  \item
    \(\hat\phi=\hat{\phi}_{X,Y}=\und{F}_{X,Y}:F\und\mcC(X;Y)\to\und\mcD(FX;FY)\)
    is the closing transformation;

  \item \(\phi^0=\overline{Fu}:\unito\to F\unito\) is a unique
    morphism such that \(\bigl[ ()\rto{u}\unito\rto{\phi^0}F\unito
    \bigr]=Fu\).
  \end{itemize}
\end{proposition}

\begin{proof} Let us check the axioms.

  \medskip

  \noindent\textbf{CF1.} By \remref{rem-C1X-C0X-iso} the equation
  \[
  \bigl[
  \unito\rto{\phi^0}F\unito\rto{Fj_X}F\und\mcC(X;X)\rto{\und{F}}\und\mcD(FX;FX)
  \bigr]=j_{FX}
  \]
  is equivalent to the equation
  \[
  \bigl[
  ()\rto{u}\unito\rto{\phi^0}F\unito\rto{Fj_X}F\und\mcC(X;X)\rto{\und{F}}\und\mcD(FX;FX)
  \bigr]=u\cdot j_{FX}=1^{\und\mcD}_{FX}.
  \]
  Since \(u\cdot\phi^0\cdot Fj_X=Fu\cdot Fj_X=F(u\cdot
  j_X)=F1^{\und\mcC}_X\), the above equation simply expresses the fact
  that the \(\mcD\)\n-functor \(\und{F}:F_*\und\mcC\to\und\mcD\)
  preserves identities, which is part of \propref{prop-und-F-D-functor}.

  \medskip

  \noindent\textbf{CF2.} The equation in question
  \[
  \bigl[
  FX\rto[F\und\mcC(u;1)^{-1}]{Fi_X}F\und\mcC(\unito;X)\rto{\und{F}}\und\mcD(F\unito;FX)
  \rto{\und\mcD(\phi^0;1)}\und\mcD(\unito;FX)
  \bigr]=i_{FX}=\und\mcD(u;1)^{-1}
  \]
  is equivalent to
  \begin{equation}
    \bigl[
    F\und\mcC(\unito;X)\rto{\und{F}}\und\mcD(F\unito;FX)\rto{\und\mcD(\phi^0;1)}
    \und\mcD(\unito;FX)\rto{\und\mcD(u;1)}\und\mcD(;FX)=FX
    \bigr]=F\und\mcC(u;1).
    \label{equ-axiom-CF2}
  \end{equation}
  The composition of the last two arrows is equal to
  \(\und\mcD(Fu;1)\). Hence the left hand side of the above equation
  is equal to
  \begin{align*}
    \bigl[
    F\und\mcC(\unito;X)&\rto{\und{F}}\und\mcD(F\unito;FX)\rto{\und\mcD(Fu;1)}\und\mcD(;FX)=FX
    \bigr]
    \\
    &=\bigl[
    F\und\mcC(\unito;X)\rto{\und{F}}\und\mcD(F\unito;FX)\rto{Fu,1}F\unito,\und\mcD(F\unito;FX)\rto{\ev^\mcD}FX
    \bigr]
    \\
    &=\bigl[
    F\und\mcC(\unito;X)\rto{Fu,1}F\unito,F\und\mcC(\unito;X)\rto{1,\und{F}}F\unito,\und\mcD(F\unito;FX)\rto{\ev^\mcD}FX
    \bigr]
    \\
    &=\bigl[
    F\und\mcC(\unito;X)\rto{Fu,1}F\unito,F\und\mcC(\unito;X)\rto{F\ev^\mcC}FX
    \bigr]=F((u,1)\cdot\ev^\mcC)=F\und\mcC(u;1).
  \end{align*}

  \noindent\textbf{CF3.} We must prove that the diagram
  \[
  \begin{xy}
    (-50,9)*+{F\und\mcC(Y;Z)}="1";
    (0,9)*+{F\und\mcC(\und\mcC(X;Y);\und\mcC(X;Z))}="2";
    (65,9)*+{\und\mcD(F\und\mcC(X;Y);F\und\mcC(X;Z))}="3";
    (-50,-9)*+{\und\mcD(FY;FZ)}="4";
    (0,-9)*+{\und\mcD(\und\mcD(FX;FY);\und\mcD(FX;FZ))}="5";
    (65,-9)*+{\und\mcD(F\und\mcC(X;Y);\und\mcD(FX;FZ))}="6";
    {\ar@{->}^-{FL^X} "1";"2"};     
    {\ar@{->}^-{\und{F}} "2";"3"};  
    {\ar@{->}_-{\und{F}} "1";"4"};
    {\ar@{->}^-{\und\mcD(1;\und{F})} "3";"6"}; 
    {\ar@{->}^-{L^{FX}} "4";"5"};
    {\ar@{->}^-{\und\mcD(\und{F};1)} "5";"6"};
  \end{xy}
  \]
  commutes. By closedness, this is equivalent to the commutativity of the
  exterior of \diaref{dia-Fmu-undF-undFundF-mu}, which expresses the fact
  that the \(\mcD\)\n-functor \(\und{F}:F_*\und\mcC\to\und\mcD\)
  preserves composition and which is part of \propref{prop-und-F-D-functor}.
  \begin{figure}
    \vspace{1 in}
    \begin{center}
      \rotatebox{90}{
        \begin{xy}
          (-90,54)*+{F\und\mcC(X;Y),F\und\mcC(Y;Z)}="1";
          (75,54)*+{\und\mcD(FX;FY),\und\mcD(FY;FZ)}="2";
          (-30,36)*+{F\und\mcC(X;Y),\und\mcD(FY;FZ)}="3";
          (-60,18)*+{\hskip1em F\und\mcC(X;Y),F\und\mcC(\und\mcC(X;Y);\und\mcC(X;Z))}="4";
          (0,0)*+{F\und\mcC(X;Y),\und\mcD(\und\mcD(FX;FY);\und\mcD(FX;FZ))}="5";
          (-30,-18)*+{F\und\mcC(X;Y),\und\mcD(F\und\mcC(X;Y);F\und\mcC(X;Z))}="6";
          (30,-36)*+{F\und\mcC(X;Y),\und\mcD(F\und\mcC(X;Y);\und\mcD(FX;FZ))}="7";
          (35,18)*+{\und\mcD(FX;FY),\und\mcD(\und\mcD(FX;FY);\und\mcD(FX;FZ))}="8";
          (-90,-54)*+{F\und\mcC(X;Z)}="9";
          (75,-54)*+{\und\mcD(FX;FZ)}="10";
          {\ar@{->}^-{\und{F},\und{F}} "1";"2"};
          {\ar@{->}^-{1,\und{F}} "1";"3"};
          {\ar@{->}_-{1,FL^X} "1";"4"};
          {\ar@{->}_-{F\mu_{\und\mcC}} "1";"9"};
          {\ar@{->}^-{1,L^{FX}} "2";"8"};
          {\ar@{->}^-{\mu_{\und\mcD}} "2";"10"};
          {\ar@{->}^-{\und{F},1} "3";"2"};
          {\ar@{->}_-{1,L^{FX}} "3";"5"};
          {\ar@{->}_-{1,\und{F}} "4";"6"};
          {\ar@{->}_-{F\ev^\mcC} "4";"9"};
          {\ar@{->}^-{1,\und\mcD(\und{F};1)} "5";"7"};
          {\ar@{->}_(.6){\und{F},1} "5";"8"};
          {\ar@{->}_(.4){1,\und\mcD(1;\und{F})} "6";"7"};
          {\ar@{->}_-{\ev^\mcD} "6";"9"};
          {\ar@{->}^-{\ev^\mcD} "7";"10"};
          {\ar@{->}^-{\ev^\mcD} "8";"10"};
          {\ar@{->}^-{\und{F}} "9";"10"};
        \end{xy}
      }
    \end{center}
    \caption{\label{dia-Fmu-undF-undFundF-mu}}
  \end{figure}

  The proposition is proven.
\end{proof}

\begin{proposition}\label{prop-multinat-cl-nat}
  A multinatural transformation \(t:F\to G:\mcC\to\mcD\) of
  multifunctors between closed multicategories with unit objects gives
  rise to a closed natural transformation given by the same
  components.
\end{proposition}

\begin{proof}
  Let
  \(\Phi=(\phi,\hat\phi,\phi^0),\Psi=(\psi,\hat\psi,\psi^0):\cc\to\cd\)
  be closed functors induced by the multifunctors \(F\) and \(G\)
  respectively. The axiom CN1 reads
  \[
  \bigl[
  \unito\rto{\phi^0}F\unito\rto{t_\unito}G\unito
  \bigr]=\psi^0.
  \]
  It is equivalent to the equation
  \[
  \bigl[
  ()\rto{u}\unito\rto{\phi^0}F\unito\rto{t_\unito}G\unito
  \bigr]=u\cdot\psi^0,
  \]
  i.e., to the equation \(Fu\cdot t_\unito=Gu\), which is a consequence
  of the multinaturality of \(t\). The axiom CN2 is a particular case of
  \propref{prop-multinatural-transformation-nuFGCD}.
\end{proof}

Let \(\ClMulticat^u\) denote the full 2\n-subcategory of
\(\ClMulticat\) whose objects are closed multicategories with a unit
object. Note that a 2\n-category is the same thing as a
\(\Cat\)\n-category. Thus we can speak about \(\Cat\)\n-functors
between 2\n-categories. These are sometimes called strict
2\n-functors; they preserve composition of 1-morphisms and
identity 1-morphisms on the nose.

\begin{proposition}\label{prop-Cat-fun-cl-multi-cl-cat}
  Propositions~\ref{prop-cl-multicat-cl-cat},
  \ref{prop-multifun-cl-fun}, and \ref{prop-multinat-cl-nat} define a
  \(\Cat\)\n-functor \(U:\ClMulticat^u\to\ClCat\).
\end{proposition}

\begin{proof}
  It is obvious that composition of 2-morphisms and identity
  2-morphisms are preserved. It is also clear that the identity
  multifunctor induces the closed identity functor. Finally,
  composition of 1-morphisms is preserved by
  \lemref{lem-underlineGF-GunderlineF-underlineG}.
\end{proof}

\section{From closed categories to closed multicategories}

In this section we prove our main result.

\begin{theorem}\label{thm-equiv}
  The \(\Cat\)\n-functor \(U:\ClMulticat^u\to\ClCat\) is a
  \(\Cat\)\n-equivalence.
\end{theorem}

We have to prove that \(U\) is bijective on 1-morphisms and
2-morphisms, and that it is essentially surjective; the latter means
that for each closed category \(\cv\) there is a closed multicategory
with a unit object such that its underlying closed category is
isomorphic (as a closed category) to \(\cv\).

\subsection{The surjectivity of \(U\) on 1-morphisms.}
Let \(\mcC\) and \(\mcD\) be closed multicategories with
unit objects.  Denote their underlying closed categories by the same
symbols. Let \(\Phi=(\phi,\hat\phi,\phi^0):\mcC\to\mcD\) be a closed
functor. We are going to define a multifunctor \(F:\mcC\to\mcD\) whose
underlying closed functor is \(\Phi\). Define \(FX=\phi X\), for each
\(X\in\Ob\mcC\). For each \(Y\in\Ob\mcC\), the map
\(F_{;Y}:\mcC(;Y)\to\mcD(;\phi Y)\) is defined via the diagram
\[
\begin{xy}
  (-30,9)*+{\mcC(;Y)}="1";
  (30,9)*+{\mcD(;\phi Y)}="2";
  (-30,-9)*+{\mcC(\unito;Y)}="3";
  (0,-9)*+{\mcD(\phi\unito;\phi Y)}="4";
  (30,-9)*+{\mcD(\unito;\phi Y)}="5";
  {\ar@{->}^-{F_{;Y}} "1";"2"};
  {\ar@{->}^-{\mcC(u;1)}_-\wr "3";"1"};
  {\ar@{->}_-{\mcD(u;1)}^-\wr "5";"2"};
  {\ar@{->}^-\phi "3";"4"};
  {\ar@{->}^-{\mcD(\phi^0;1)} "4";"5"};
\end{xy}
\]
Recall that for a morphism \(f:()\to Y\) we denote by
\(\overline{f}:\unito\to Y\) a unique morphism such that
\(u\cdot\overline{f}=f\).  Then the commutativity in the above diagram
means that
\begin{equation}
  Ff=\bigl[
  ()\rto{u}\unito\rto{\phi^0}\phi\unito\rto{\phi(\overline{f})}\phi Y
  \bigr],
  \label{equ-def-F-empty-source}
\end{equation}
for each \(f:()\to Y\). For \(n\ge1\) and
\(X_1,\dots,X_n,Y\in\Ob\mcC\), the map
\[
F_{X_1,\dots,X_n;Y}:\mcC(X_1,\dots,X_n;Y)\to\mcD(\phi X_1,\dots,\phi X_n;\phi Y)
\]
is defined inductively by requesting the commutativity in the diagram
\begin{equation}
  \begin{xy}
    (-40,18)*+{\mcC(X_2,\dots,X_n;\und\mcC(X_1;Y))}="1";
    (40,18)*+{\mcD(\phi X_2,\dots,\phi X_n;\phi\und\mcC(X_1;Y))}="2";
    (40,0)*+{\mcD(\phi X_2,\dots,\phi X_n;\und\mcD(\phi X_1;\phi Y))}="3";
    (-40,-18)*+{\mcC(X_1,\dots,X_n;Y)}="4";
    (40,-18)*+{\mcD(\phi X_1,\dots,\phi X_n;\phi Y)}="5";
    {\ar@{->}^-{F_{X_2,\dots,X_n;\und\mcC(X_1;Y)}} "1";"2"};
    {\ar@{->}_-{\varphi^\mcC}^-\wr "1";"4"};
    {\ar@{->}^-{\mcD(1;\hat\phi)} "2";"3"};
    {\ar@{->}^-{\varphi^\mcD}_-\wr "3";"5"};
    {\ar@{->}^-{F_{X_1,\dots,X_n;Y}} "4";"5"};
  \end{xy}
  \label{dia-def-F}
\end{equation}

\begin{lemma}\label{lem-claim-1}
  The following diagram commutes
  \[
  \begin{xy}
    (-40,9)*+{\mcC(;\und\mcC(X;Y))}="1";
    (0,9)*+{\mcD(;\phi\und\mcC(X;Y))}="2";
    (40,9)*+{\mcD(;\und\mcD(\phi X;\phi Y))}="3";
    (-40,-9)*+{\mcC(X;Y)}="4";
    (40,-9)*+{\mcD(\phi X;\phi Y)}="5";
    {\ar@{->}^-{F_{;\und\mcC(X;Y)}} "1";"2"};
    {\ar@{->}^-{\mcD(;\hat\phi)} "2";"3"};
    {\ar@{->}_-{\varphi^\mcC}^-\wr "1";"4"};
    {\ar@{->}^-{\varphi^\mcD}_-\wr "3";"5"};
    {\ar@{->}^-\phi "4";"5"};
  \end{xy}
  \]
  In particular, \(F_{X;Y}=\phi_{X,Y}:\mcC(X;Y)\to\mcD(\phi X;\phi Y)\).
\end{lemma}

\begin{proof}
  Equivalently, the exterior of the diagram
  \[
  \begin{xy}
    (-40,18)*+{\mcC(;\und\mcC(X;Y))}="1";
    (40,18)*+{\mcD(;\phi\und\mcC(X;Y))}="2";
    (80,18)*+{\mcD(;\und\mcD(\phi X;\phi Y))}="3";
    (-40,0)*+{\mcC(\unito;\und\mcC(X;Y))}="4";
    (-5,0)*+{\mcD(\phi\unito;\phi\und\mcC(X;Y))}="5";
    (40,0)*+{\mcD(\unito;\phi\und\mcC(X;Y))}="6";
    (80,0)*+{\mcD(\unito;\und\mcD(\phi X;\phi Y))}="7";
    (-40,-18)*+{\mcC(X;Y)}="8";
    (80,-18)*+{\mcD(\phi X;\phi Y)}="9";
    {\ar@{->}^-{F_{;\und\mcC(X;Y)}} "1";"2"};
    {\ar@{->}^-{\mcD(;\hat\phi)} "2";"3"};
    {\ar@{->}_-{\mcC(u;1)}^-\wr "4";"1"};
    {\ar@{->}^-{\mcD(u;1)}_-\wr "6";"2"};
    {\ar@{->}^-{\mcD(u;1)}_-\wr "7";"3"};
    {\ar@{->}^-\phi "4";"5"};
    {\ar@{->}^-{\mcD(\phi^0;1)} "5";"6"};
    {\ar@{->}^-{\mcD(1;\hat\phi)} "6";"7"};
    {\ar@{->}_-\gamma "8";"4"};
    {\ar@{->}^-\gamma "9";"7"};
    {\ar@{->}^-\phi "8";"9"};
    {\ar@{->}@/^{4pc}/_(.3){(\varphi^\mcC)^{-1}} "8";"1"};
    {\ar@{->}@/_{4pc}/^(.3){(\varphi^\mcD)^{-1}} "9";"3"};
  \end{xy}
  \]
  commutes. The upper pentagon is the definition of
  \(F_{;\und\mcC(X;Y)}\). The bottom hexagon commutes. Indeed, taking
  \(f\in\mcC(X;Y)\) and tracing it along the left-top path yields
  \begin{alignat*}{3}
    \phi^0\cdot\phi(j_X)\cdot\phi\und\mcC(1;f)\cdot\hat\phi&=
    \phi^0\cdot\phi(j_X)\cdot\hat\phi\cdot\mcD(1;\phi(f))&\quad&\textup{(naturality of \(\hat\phi\))}
    \\
    &=j_{\phi X}\cdot\und\mcD(1;\phi(f)),&\quad&\textup{(axiom CF1)}
  \end{alignat*}
  which is precisely the image of \(f\) along the bottom-right path.
\end{proof}

\begin{lemma}\label{lem-claim-2}
  For each \(f:()\to Y\) and \(Z\in\Ob\mcC\), the diagram
  \[
  \begin{xy}
    (-20,9)*+{\phi\und\mcC(Y;Z)}="1";
    (20,9)*+{\phi\und\mcC(;Z)}="2";
    (35,9)*+{\phi Z}="3";
    (-20,-9)*+{\und\mcD(\phi Y;\phi Z)}="4";
    (20,-9)*+{\und\mcD(;\phi Z)}="5";
    (35,-9)*+{\phi Z}="6";
    {\ar@{->}^-{\phi\und\mcC(f;1)} "1";"2"};
    {\ar@{=} "2";"3"};
    {\ar@{->}_-{\hat\phi} "1";"4"};
    {\ar@{=} "3";"6"};
    {\ar@{->}^-{\und\mcD(Ff;1)} "4";"5"};
    {\ar@{=} "5";"6"};
  \end{xy}
  \]
  commutes.
\end{lemma}

\begin{proof}
  By definition,
  \[
  Ff=\bigl[
  ()\rto{u}\unito\rto{\phi^0}\phi\unito\rto{\phi(\overline{f})}\phi Y
  \bigr].
  \]
  The diagram
  \[
  \begin{xy}
    (-40,9)*+{\phi\und\mcC(Y;Z)}="1";
    (0,9)*+{\phi\und\mcC(\unito;Z)}="2";
    (40,9)*+{\phi\und\mcC(;Z)}="3";
    (55,9)*+{\phi Z}="4";
    (-40,-9)*+{\und\mcD(\phi Y;\phi Z)}="5";
    (0,-9)*+{\und\mcD(\phi\unito;\phi Z)}="6";
    (40,-9)*+{\und\mcD(;\phi Z)}="7";
    (55,-9)*+{\phi Z}="8";
    {\ar@{->}^-{\phi\und\mcC(\overline{f};1)} "1";"2"};
    {\ar@{->}^-{\phi\und\mcC(u;1)} "2";"3"};
    {\ar@{=} "3";"4"};
    {\ar@{->}_-{\hat\phi} "1";"5"};
    {\ar@{->}^-{\hat\phi} "2";"6"};
    {\ar@{=} "4";"8"};
    {\ar@{->}^-{\und\mcD(\phi(\overline{f});1)} "5";"6"};
    {\ar@{->}^-{\und\mcD(u\cdot\phi^0;1)} "6";"7"};
    {\ar@{=} "7";"8"};
    {\ar@{->}@/^{2.5pc}/^-{\phi\und\mcC(f;1)} "1";"4"};
    {\ar@{->}@/_{2.5pc}/^-{\und\mcD(Ff;1)} "5";"8"};
  \end{xy}
  \]
  commutes. Indeed, the left square commutes by the naturality of
  \(\hat\phi\), while the commutativity of the right square is a
  consequence of the axiom CF2, see~\eqref{equ-axiom-CF2}.
\end{proof}

With the notation of \lemref{lem-aux-identities}, we can rewrite the
commutativity condition in diagram~\eqref{dia-def-F} as a recursive
formula for the multigraph morphism \(F\):
\[
Ff=\varphi^\mcD(F((\varphi^\mcC)^{-1}(f))\cdot\hat\phi)
=\varphi^\mcD(F\langle f\rangle\cdot\hat\phi),
\]
for each \(f:X_1,\dots,X_n\to Y\) with \(n\ge1\), or equivalently
\begin{equation}
  \langle Ff\rangle=\bigl[ \phi X_2,\dots,\phi
  X_n\rto{F\langle f\rangle}\phi\und\mcC(X_1;Y)
  \rto{\hat\phi}\und\mcD(\phi X_1;\phi Y)
  \bigr].
  \label{equ-def-F}
\end{equation}

\begin{lemma}\label{lem-claim-3}
  For each \(X,Y,Z\in\Ob\mcC\), the diagram
  \[
  \begin{xy}
    (-25,9)*+{\phi\und\mcC(X;Y),\phi\und\mcC(Y;Z)}="1";
    (25,9)*+{\phi\und\mcC(X;Z)}="2";
    (-25,-9)*+{\und\mcD(\phi X;\phi Y),\und\mcD(\phi Y;\phi Z)}="3";
    (25,-9)*+{\und\mcD(\phi X;\phi Z)}="4";
    {\ar@{->}^-{F\mu_{\und\mcC}} "1";"2"};
    {\ar@{->}_-{\hat\phi,\hat\phi} "1";"3"};
    {\ar@{->}^-{\hat\phi} "2";"4"};
    {\ar@{->}^-{\mu_{\und\mcD}} "3";"4"};
  \end{xy}
  \]
  commutes.
\end{lemma}

\begin{proof}
  It suffices to prove the equation
  \[
  \langle F\mu_{\und\mcC}\cdot\hat\phi\rangle
  =\langle(\hat\phi,\hat\phi)\cdot\mu_{\und\mcD}\rangle.
  \]
  By \lemref{lem-aux-identities},(c), the left hand side is equal to
  \[
  \phi\und\mcC(Y;Z)\rto{\langle F\mu_{\und\mcC}\rangle}\und\mcD(\phi\und\mcC(X;Y);\phi\und\mcC(X;Z))
  \rto{\und\mcD(1;\hat\phi)}\und\mcD(\phi\und\mcC(X;Y);\und\mcD(\phi X;\phi Z)),
  \]
  while the right hand side is equal to
  \begin{equation*}
    \phi\und\mcC(Y;Z)\rto{\hat\phi}\und\mcD(\phi Y;\phi Z)
    \rto{\langle\mu_{\und\mcD}\rangle}
    \und\mcD(\und\mcD(\phi X;\phi Y);\und\mcD(\phi Y;\phi Z))
    \rto{\und\mcD(\hat\phi;1)}
    \und\mcD(\phi\und\mcC(X;Y);\und\mcD(\phi X;\phi Z))
  \end{equation*}
  by \lemref{lem-aux-identities},(b). Note that
  \(\langle\mu_{\und\mcD}\rangle=(\varphi^\mcD)^{-1}(\mu_{\und\mcD})=L^{\phi X}\).
  Furthermore, by \eqref{equ-def-F},
  \begin{align*}
    \langle F\mu_{\und\mcC}\rangle&=\bigl[
    \phi\und\mcC(Y;Z)\rto{\phi\langle\mu_{\und\mcC}\rangle}
    \phi\und\mcC(\und\mcC(X;Y);\und\mcC(X;Z))
    \rto{\hat\phi}\und\mcD(\phi\und\mcC(X;Y);\phi\und\mcC(X;Z)) \bigr]
    \\
    &=\bigl[ \phi\und\mcC(Y;Z)\rto{\phi
      L^X}\phi\und\mcC(\und\mcC(X;Y);\und\mcC(X;Z))
    \rto{\hat\phi}\und\mcD(\phi\und\mcC(X;Y);\phi\und\mcC(X;Z))
    \bigr],
  \end{align*}
  therefore the equation in question is simply the axiom CF3.
\end{proof}

\begin{proposition}
  The multigraph morphism \(F:\mcC\to\mcD\) is a multifunctor, and its
  underlying closed functor is \(\Phi\).
\end{proposition}

\begin{proof}
  Trivially, \(F\) preserves identities since so does \(\phi\). Let us
  prove that \(F\) preserves composition. The proof is in three steps.

  \begin{lemma}\label{lem-claim-5}
    \(F\) preserves composition of the form
    \(X_1,\dots,X_k\rto{f}Y\rto{g}Z\).
  \end{lemma}

  \begin{proof}
    The proof is by induction on \(k\). There is nothing to prove in the
    case \(k=1\). Suppose that \(k=0\) and we are given composable
    morphisms
    \[
    ()\rto{f}X\rto{g}Y.
    \]
    Then since \(u\cdot\overline{fg}=f\cdot g=(u\cdot\overline{f})\cdot g=u\cdot(\overline{f}\cdot
    g)\), it follows that \(\overline{f\cdot g}=\overline{f}\cdot g\). By
    formula~\eqref{equ-def-F-empty-source},
    \[
    F(f\cdot g)=u\cdot\phi^0\cdot\phi(\overline{f\cdot
      g})=u\cdot\phi^0\cdot\phi(\overline{f}\cdot
    g)=u\cdot\phi^0\cdot\phi(\overline{f})\cdot\phi(g)=Ff\cdot Fg.
    \]
    Suppose that \(k>1\). Then
    \begin{alignat*}{3}
      \langle F(f\cdot g)\rangle&=F\langle f\cdot
      g\rangle\cdot\hat\phi & \quad &
      \textup{(formula~\eqref{equ-def-F})}
      \\
      &=F(\langle f\rangle\cdot\und\mcC(1;g))\cdot\hat\phi & \quad &
      \textup{(\lemref{lem-aux-identities},(c))}
      \\
      &=F\langle f\rangle\cdot\phi\und\mcC(1;g)\cdot\hat\phi & \quad &
      \textup{(induction hypothesis)}
      \\
      &=F\langle f\rangle\cdot\hat\phi\cdot\und\mcD(1;\phi(g)) & \quad &
      \textup{(naturality of \(\hat\phi\))}
      \\
      &=\langle Ff\rangle\cdot\und\mcD(1;Fg) & \quad &
      \textup{(formula~\eqref{equ-def-F})}
      \\
      &=\langle Ff\cdot Fg\rangle, & \quad &
      \textup{(\lemref{lem-aux-identities},(c))}
    \end{alignat*}
    and induction goes through.
  \end{proof}

  \begin{lemma}\label{lem-claim-6}
    \(F\) preserves composition of the form
    \(X^1_1,\dots,X^{k_1}_1,X^1_2,\dots,X^{k_2}_2\rto{f_1,f_2}Y_1,Y_2\rto{g}Z\).
  \end{lemma}

  \begin{proof}
    The proof is by induction on \(k_1\). If \(k_1=0\), then by
    \lemref{lem-aux-identities},(a),
    \[
    (f_1,f_2)\cdot g=\bigl[
    X^1_2,\dots,X^{k_2}_2\rto{f_2}Y_2\rto{\langle g\rangle}
    \und\mcC(Y_1;Z)\rto{\und\mcC(f_1;1)}\und\mcC(;Z)=Z \bigr],
    \]
    therefore
    \begin{alignat*}{3}
      F((f_1,f_2)\cdot g)&=Ff_2\cdot \phi\langle g\rangle\cdot
      \phi\und\mcC(f_1;1) & \quad & \textup{(\lemref{lem-claim-5})}
      \\
      &=Ff_2\cdot\phi\langle g\rangle\cdot\hat\phi\cdot\und\mcD(\phi(f_1);1) &
      \quad & \textup{(\lemref{lem-claim-2})}
      \\
      &=Ff_2\cdot\langle Fg\rangle\cdot\und\mcD(Ff_1;1) & \quad &
      \textup{(formula~\eqref{equ-def-F})}
      \\
      &=(Ff_1,Ff_2)\cdot Fg. & \quad &
      \textup{(\lemref{lem-aux-identities},(a))}
    \end{alignat*}
    If \(k_1=1\), then by \lemref{lem-aux-identities},(b),
    \[
    \langle(f_1,f_2)\cdot g\rangle=\bigl[
    X^1_2,\dots,X^{k_2}_2\rto{f_2}Y_2\rto{\langle g\rangle}
    \und\mcC(Y_1;Z)\rto{\und\mcC(f_1;1)}\und\mcC(X^1_1;Z)
    \bigr],
    \]
    therefore
    \begin{alignat*}{3}
      \langle F((f_1,f_2)\cdot g)\rangle &=F\langle(f_1,f_2)\cdot
      g\rangle\cdot\hat\phi & \quad &
      \textup{(formula~\eqref{equ-def-F})}
      \\
      &=Ff_2\cdot\phi\langle
      g\rangle\cdot\phi\und\mcC(f_1;1)\cdot\hat\phi 
      & \quad & \textup{(\lemref{lem-claim-5})}
      \\
      &=Ff_2\cdot\phi\langle g\rangle\cdot\hat\phi\cdot\und\mcD(\phi(f_1);1) &
      \quad & \textup{(naturality of \(\hat\phi\))}
      \\
      &=Ff_2\cdot\langle Fg\rangle\cdot\und\mcD(Ff_1;1) & \quad &
      \textup{(formula~\eqref{equ-def-F})}
      \\
      &=\langle(Ff_1,Ff_2)\cdot Fg\rangle, & \quad &
      \textup{(\lemref{lem-aux-identities},(b))}
    \end{alignat*}
    and hence \(F((f_1,f_2)\cdot g)=(Ff_1,Ff_2)\cdot Fg\). Suppose that
    \(k_1>1\).  Then by \lemref{lem-aux-identities},(c)
    \(\langle(f_1,f_2)\cdot g\rangle\) is equal to
    \begin{equation*}
      \bigl[
      X^2_1,\dots,X^{k_1}_1,X^1_2,\dots,X^{k_2}_2\rto{\langle
        f_1\rangle,f_2}
      \und\mcC(X^1_1;Y_1),Y_2
      \rto{1,\langle g\rangle}\und\mcC(X^1_1;Y_1),\und\mcC(Y_1;Z)
      \rto{\mu_{\und\mcC}}\und\mcC(X^1_1;Z)
      \bigr],
    \end{equation*}
    therefore
    \begin{alignat*}{3}
      \langle F((f_1,f_2)\cdot g)\rangle &=F\langle(f_1,f_2)\cdot
      g\rangle\cdot\hat\phi & \quad &
      \textup{(formula~\eqref{equ-def-F})}
      \\
      &=(F\langle f_1\rangle,Ff_2)\cdot F((1,\langle
      g\rangle)\mu_{\und\mcC})\cdot\hat\phi & \quad &
      \textup{(induction hypothesis)}
      \\
      &=(F\langle f_1\rangle,Ff_2)\cdot(1,F\langle g\rangle)\cdot
      F\mu_{\und\mcC}\cdot\hat\phi & \quad & \textup{(case \(k_1=1\))}
      \\
      &=(F\langle f_1\rangle,Ff_2)\cdot(1,F\langle
      g\rangle)\cdot(\hat\phi,\hat\phi)\cdot\mu_{\und\mcD} & \quad &
      \textup{(\lemref{lem-claim-3})}
      \\
      &=(F\langle f_1\rangle\cdot\hat\phi,Ff_2)\cdot(1,F\langle
      g\rangle\cdot\hat\phi)\cdot\mu_{\und\mcD}
      \\
      &=(\langle Ff_1\rangle,Ff_2)\cdot(1,\langle
      Fg\rangle)\cdot\mu_{\und\mcD} & \quad &
      \textup{(formula~\eqref{equ-def-F})}
      \\
      &=\langle(Ff_1,Ff_2)\cdot Fg\rangle, & \quad &
      \textup{(\lemref{lem-aux-identities},(c))}
    \end{alignat*}
    hence \(F((f_1,f_2)\cdot g)=(Ff_1,Ff_2)\cdot Fg\), and the lemma is
    proven.
  \end{proof}

  \begin{lemma}
    \(F\) preserves composition of the form
    \begin{equation}
      X^1_1,\dots,X^{k_1}_1,\dots,X^1_n,\dots,X^{k_n}_n\rto{f_1,\dots,f_n}Y_1,\dots,Y_n\rto{g}Z.
      \label{equ-f1-fn-g}
    \end{equation}
  \end{lemma}

  \begin{proof}
    The proof is by induction on \(n\), and for a fixed \(n\) by
    induction on \(k_1\). We have worked out the cases \(n=1\) and
    \(n=2\) explicitly in Lemmas~\ref{lem-claim-5} and
    \ref{lem-claim-6}. Assume that \(F\) preserves an arbitrary
    composition of the form
    \[
    U^1_1,\dots,U^{l_1}_1,\dots,U^1_{n-1},\dots,U^{l_{n-1}}_{n-1}\rto{p_1,\dots,p_{n-1}}V_1,\dots,V_{n-1}\rto{q}W,
    \]
    and suppose we are given composite~\eqref{equ-f1-fn-g}. We do
    induction on \(k_1\). If \(k_1=0\), then by
    \lemref{lem-aux-identities},(a) \((f_1,\dots,f_n)\cdot g\) is equal
    to
    \begin{equation*}
      \bigl[
      X^1_2,\dots,X^{k_2}_2,\dots,X^1_n,\dots,X^{k_n}_n\rto{f_2,\dots,f_n}
      Y_2,\dots,Y_n\rto{\langle g\rangle}
      \und\mcC(Y_1;Z)\rto{\und\mcC(f_1;1)}\und\mcC(;Z)=Z
      \bigr],
    \end{equation*}
    therefore
    \begin{alignat*}{3}
      F((f_1,\dots,f_n)\cdot g)&=(Ff_2,\dots,Ff_n)\cdot F(\langle
      g\rangle\cdot\und\mcC(f_1;1)) & \quad & \textup{(induction
        hypothesis)}
      \\
      &=(Ff_2,\dots,Ff_n)\cdot(F\langle
      g\rangle\cdot\phi\und\mcC(f_1;1)) & \quad &
      \textup{(\lemref{lem-claim-5})}
      \\
      &=(Ff_2,\dots,Ff_n)\cdot(F\langle
      g\rangle\cdot\hat\phi\cdot\und\mcD(\phi(f_1);1)) & \quad &
      \textup{(\lemref{lem-claim-2})}
      \\
      &=(Ff_2,\dots,Ff_n)\cdot(\langle Fg\rangle\cdot\und\mcD(Ff_1;1))
      & \quad & \textup{(formula~\eqref{equ-def-F})}
      \\
      &=(Ff_1,\dots,Ff_n)\cdot Fg. & \quad &
      \textup{(\lemref{lem-aux-identities},(a))}
    \end{alignat*}
    Suppose that \(k_1=1\). Then by \lemref{lem-aux-identities},(b)
    \(\langle(f_1,\dots,f_n)\cdot g\rangle\) is equal to
    \begin{equation*}
      \bigl[
      X^1_2,\dots,X^{k_2}_2,\dots,X^1_n,\dots,X^{k_n}_n\rto{f_2,\dots,f_n}
      Y_2,\dots,Y_n\rto{\langle g\rangle}
      \und\mcC(Y_1;Z)\rto{\und\mcC(f_1;1)}\und\mcC(X^1_1;Z)
      \bigr],
    \end{equation*}
    therefore
    \begin{alignat*}{3}
      \langle F((f_1,\dots,f_n)\cdot
      g)\rangle&=F\langle(f_1,\dots,f_n)\cdot g\rangle\cdot\hat\phi &
      \quad & \textup{(formula~\eqref{equ-def-F})}
      \\
      &=(Ff_2,\dots,Ff_n)\cdot F(\langle
      g\rangle\cdot\und\mcC(f_1;1))\cdot\hat\phi & \quad &
      \textup{(induction hypothesis)}
      \\
      &=(Ff_2,\dots,Ff_n)\cdot
      F\langle g\rangle\cdot\phi\und\mcC(f_1;1)\cdot\hat\phi & \quad &
      \textup{(\lemref{lem-claim-5})}
      \\
      &=(Ff_2,\dots,Ff_n)\cdot
      F\langle g\rangle\cdot\hat\phi\cdot\und\mcD(\phi(f_1);1) & \quad &
      \textup{(naturality of \(\hat\phi\))}
      \\
      &=(Ff_2,\dots,Ff_n)\cdot\langle Fg\rangle\cdot\und\mcD(Ff_1;1) & \quad &
      \textup{(formula~\eqref{equ-def-F})}
      \\
      &=\langle(Ff_1,\dots,Ff_n)\cdot Fg\rangle, & \quad &
      \textup{(\lemref{lem-aux-identities},(b))}
    \end{alignat*}
    and hence \(F((f_1,\dots,f_n)\cdot g)=(Ff_1,\dots,Ff_n)\cdot
    Fg\). Suppose that \(k_1>1\), then by
    \lemref{lem-aux-identities},(c) \(\langle(f_1,\dots,f_n)\cdot
    g\rangle\) is equal to
    \begin{align*}
      X^2_1,\dots,X^{k_1}_1,X^1_2,\dots,X^{k_2}_2,\dots,X^1_n,\dots,X^{k_n}_n
      &\rto{\langle f_1\rangle,f_2,\dots,f_n}
      \und\mcC(X^1_1;Y_1),Y_2,\dots,Y_n
      \\
      &\rto[\hphantom{\langle f_1\rangle,f_2,\dots,f_n}]{1,\langle
        g\rangle}\und\mcC(X^1_1;Y_1),\und\mcC(Y_1;Z)
      \\
      &\rto[\hphantom{\langle
        f_1\rangle,f_2,\dots,f_n}]{\mu_{\und\mcC}}\und\mcC(X^1_1;Z),
    \end{align*}
    therefore
    \begin{alignat*}{3}
      \langle F((f_1,\dots,f_n)\cdot g)\rangle &=
      F\langle(f_1,\dots,f_n)\cdot g\rangle\cdot\hat\phi & \quad &
      \textup{(formula~\eqref{equ-def-F})}
      \\
      &=(F\langle f_1\rangle,Ff_2,\dots,Ff_n)\cdot F((1,\langle
      g\rangle)\mu_{\und\mcC})\cdot\hat\phi & \quad &
      \textup{(induction hypothesis)}
      \\
      &=(F\langle f_1\rangle,Ff_2,\dots,Ff_n)\cdot(1,F[g])\cdot
      F\mu_{\und\mcC}\cdot\hat\phi & \quad &
      \textup{(\lemref{lem-claim-6})}
      \\
      &=(F\langle f_1\rangle,Ff_2,\dots,Ff_n)\cdot(1,F\langle
      g\rangle)\cdot(\hat\phi,\hat\phi)\cdot\mu_{\und\mcD} & \quad &
      \textup{(\lemref{lem-claim-3})}
      \\
      &=(F\langle
      f_1\rangle\cdot\hat\phi,Ff_2,\dots,Ff_n)\cdot(1,F\langle
      g\rangle\cdot\hat\phi)\cdot\mu_{\und\mcD} & \quad &
      \\
      &=(\langle Ff_1\rangle,Ff_2,\dots,Ff_n)\cdot(1,\langle
      Fg\rangle)\cdot\mu_{\und\mcD} & \quad &
      \textup{(formula~\eqref{equ-def-F})}
      \\
      &=\langle(Ff_1,\dots,Ff_n)\cdot Fg\rangle, & \quad &
      \textup{(\lemref{lem-aux-identities},(c))}
    \end{alignat*}
    hence \(F((f_1,\dots,f_n)\cdot g)=(Ff_1,\dots,Ff_n)\cdot Fg\), and
    induction goes through.
  \end{proof}

  Thus we have proven that \(F:\mcC\to\mcD\) is a multifunctor. By
  construction, its underlying functor is \(\phi\). Furthermore, the
  closing transformation \(\und{F}_{X;Y}\) coincides with
  \(\hat{\phi}_{X,Y}:\phi\und\mcC(X;Y)\to\und\mcD(\phi X;\phi
  Y)\). Indeed, notice that \(\und{F}_{X,Y}=\langle
  F\ev^\mcC\rangle\), where \(\ev^\mcC:X,\und\mcC(X;Y)\to Y\) is the
  evaluation morphism. Further, by formula~\eqref{equ-def-F},
  \[
  \und{F}_{X,Y}=\langle
  F\ev^\mcC\rangle=\phi\langle\ev^\mcC\rangle\cdot\hat{\phi}_{X,Y}
  =\hat{\phi}_{X,Y},
  \]
  since \(\langle\ev^\mcC\rangle = 1
  :\und\mcC(X;Y)\to\und\mcC(X;Y)\). Finally,
  \[
  Fu=\bigl[
  ()\rto{u}\unito\rto{\phi^0}\phi\unito
  \bigr].
  \]
  Indeed, by formula~\eqref{equ-def-F-empty-source},
  \[
  Fu=\bigl[
  ()\rto{u}\unito\rto{\phi^0}\phi\unito\rto{\phi(\overline{u})}\phi\unito
  \bigr]
  =\bigl[
  ()\rto{u}\unito\rto{\phi^0}\phi\unito
  \bigr],
  \]
  since \(\overline{u}=1:\unito\to\unito\). Thus we conclude that
  \(F:\mcC\to\mcD\) is a multifunctor whose underlying closed functor is
  \(\Phi\). The proposition is proven.
\end{proof}

\subsection{The injectivity of \(U\) on 1\n-morphisms.} The following
proposition shows that the \(\Cat\)\n-functor \(U\) is injective on
1\n-morphisms.

\begin{proposition}
  Let \(F,G:\mcC\to\mcD\) be multifunctors between closed multicategories
  with unit objects. Suppose that \(F\) and \(G\) induce the same closed
  functor \(\Phi=(\phi,\hat\phi,\phi^0)\) between the underlying closed
  categories. Then \(F=G\).
\end{proposition}

\begin{proof}
  By assumption, the underlying functors of the multifunctors \(F\) and
  \(G\) are the same and are equal to the functor \(\phi\). Let us prove
  that \(Ff=Gf\), for each \(f:X_1,\dots,X_n\to Y\). The proof is by
  induction on \(n\). There is nothing to prove if \(n=1\). Suppose that
  \(n=0\), i.e., \(f\) is a morphism \(()\to Y\). Then since \(F\) and
  \(G\) are multifunctors,
  \[
  Ff=F(u\cdot\overline{f})=Fu\cdot F\overline{f},\qquad
  Gf=G(u\cdot\overline{f})=Gu\cdot G\overline{f}.
  \]
  Since \(F\) and \(G\) coincide on morphisms with one source object, it
  follows that \(F\overline{f}=G\overline{f}\). Furthermore,
  \[
  Fu=\bigl[
  ()\rto{u}\unito\rto{\phi^0}F\unito=G\unito
  \bigr]=Gu,
  \]
  hence \(Ff=Gf\). The induction step follows from the commutative
  diagram
  \[
  \begin{xy}
    (-30,18)*+{\mcC(X_2,\dots,X_n;\und\mcC(X_1;Y))}="1";
    (30,18)*+{\mcD(\phi X_2,\dots,\phi X_n;\phi\und\mcC(X_1;Y))}="2";
    (30,0)*+{\mcD(\phi X_2,\dots,\phi X_n;\und\mcD(\phi X_1;\phi Y))}="3";
    (-30,-18)*+{\mcC(X_1,\dots,X_n;Y)}="4";
    (30,-18)*+{\mcD(\phi X_1,\dots,\phi X_n;\phi Y)}="5";
    {\ar@{->}^-{F} "1";"2"};
    {\ar@{->}_-{\varphi^\mcC}^-\wr "1";"4"};
    {\ar@{->}^-{\mcD(1;\hat\phi)} "2";"3"};
    {\ar@{->}^-{\varphi^\mcD}_-\wr "3";"5"};
    {\ar@{->}^-{F} "4";"5"};
  \end{xy}
  \]
  and a similar diagram for \(G\), which are particular cases of
  \propref{prop-F-D1undF-phi-phi-F}.
\end{proof}

\subsection{The bijectivity of \(U\) on 2-morphisms.} 
The following proposition implies that \(U\) is bijective on
2\n-morphisms.

\begin{proposition}
  Let \(F,G:\mcC\to\mcD\) be multifunctors between closed
  multicategories with unit objects. Denote by
  \(\Phi=(\phi,\hat\phi,\phi^0)\) and \(\Psi=(\psi,\hat\psi,\psi^0)\)
  the corresponding closed functors. Let \(r:\Phi\to\Psi\) be a closed
  natural transformation. Then \(r\) is also a multinatural
  transformation \(F\to G:\mcC\to\mcD\).
\end{proposition}

\begin{proof}
  We must prove that, for each \(f:X_1,\dots,X_n\to Y\), the equation
  \[
  Ff\cdot r_Y=(r_{X_1},\dots,r_{X_n})\cdot Gf
  \]
  holds true. The proof is by induction on \(n\). Suppose that
  \(n=0\), and that \(f\) is a morphism \(()\to Y\). The axiom CN1
  \[
  \bigl[ \unito\rto{\phi^0}F\unito\rto{r_\unito}G\unito \bigr]=\psi^0
  \]
  implies
  \[
  \bigl[ ()\rto{Fu}F\unito\rto{r_\unito}G\unito \bigr]=Gu.
  \]
  It follows that
  \[
  Ff\cdot r_Y=Fu\cdot F\overline{f}\cdot r_Y=Fu\cdot r_\unito\cdot
  G\overline{f}=Gu\cdot G\overline{f}=Gf,
  \]
  where the second equality is due to the naturality of \(r\). There
  is nothing to prove in the case \(n=1\). Suppose that \(n>1\). It
  suffices to prove that
  \[ \langle Ff\cdot r_Y\rangle=\langle(r_{X_1},\dots,r_{X_n})\cdot
  Gf\rangle:FX_2,\dots,FX_n\to\und\mcD(FX_1;GY).
  \]
  By \lemref{lem-aux-identities},(c), the left hand side expands out
  as \(\langle Ff\rangle\cdot\und\mcD(1;r_Y)\), which by
  formula~\eqref{equ-def-F} is equal to \(F\langle
  f\rangle\cdot\hat\phi\cdot\und\mcD(1;r_Y)\). By
  \lemref{lem-aux-identities},(b), the right hand side of the equation
  in question is equal to \((r_{X_2},\dots,r_{X_n})\cdot\langle
  Gf\rangle\cdot\und\mcD(r_{X_1};1)\), which by
  formula~\eqref{equ-def-F} is equal to \((r_{X_2},\dots,r_{X_n})\cdot
  G\langle f\rangle\cdot\hat\psi\cdot\und\mcD(r_{X_1};1)\). By the induction
  hypothesis, the latter is equal to
  \[
  F\langle f\rangle\cdot
  r_{\und\mcC(X_1;Y)}\cdot\hat\psi\cdot\und\mcD(r_{X_1};1).
  \]
  The required equation follows then from the axiom CN2.
\end{proof}

\subsection{The essential surjectivity of \(U\).}

Let us prove that for each closed category \(\cv\) there is a closed
multicategory \(\mcV\) with a unit object whose underlying closed
category is isomorphic to \(\cv\). First of all, notice that by
\thmref{thm-cl-cat-iso-EK} we may (and we shall) assume in what
follows that \(\cv\) is a closed category in the sense of Eilenberg
and Kelly; i.e., that \(\cv\) is equipped with a functor
\(V:\cv\to\Set\) such that
\(V\und\cv(-,-)=\cv(-,-):\cv^\op\times\cv\to\Set\) and the axiom CC5'
is satisfied. In particular, we can use the whole theory of closed
categories developed in \cite{EK} without any modifications. We are
now going to construct a closed multicategory \(\mcV\) with a unit
object whose underlying closed category is isomorphic to \(\cv\).  The
construction is based on ideas of Laplaza's paper \cite{Laplaza}.

We recall that for each object \(X\) of the category \(\cv\) we have a
\(\cv\)\n-functor \(L^X:\und\cv\to\und\cv\), and for each \(f\in
V\und\cv(X,Y)=\cv(X,Y)\) there is a \(\cv\)\n-natural transformation
\(L^f:L^Y\to L^X:\und\cv\to\und\cv\), uniquely determined by the
condition \((V(L^f)_Y)1_Y=f\), see Examples~\ref{ex-V-LX},
\ref{exa-Lf}, \ref{ex-Lf-repr}, or \cite[Section~9]{EK}. Moreover, by
\cite[Proposition~9.2]{EK} the assignments \(X\mapsto L^X\) and
\(f\mapsto L^f\) determine a fully faithful functor from the category
\(\cv^\op\) to the category \(\VCat(\und\cv,\und\cv)\) of
\(\cv\)\n-functors \(\und\cv\to\und\cv\) and their \(\cv\)\n-natural
transformations. For us it is more convenient to write it as functor
from \(\cv\) to \(\VCat(\und\cv,\und\cv)^\op\). Note that the latter
category is strict monoidal with the tensor product given by
composition of \(\cv\)\n-functors. More precisely, the tensor product
of \(F\) and \(G\) in the given order is \(FG=F\cdot G=G\circ
F\). Consider the multicategory associated with
\(\VCat(\und\cv,\und\cv)^\op\) (see \exaref{ex:monoidal-cat-multicat})
and consider its full submulticategory whose objects are
\(\cv\)\n-functors \(L^X\), \(X\in\Ob\cv\). That is, in essence, our
\(\mcV\). More precisely, \(\Ob\mcV=\Ob\cv\) and
\[
\mcV(X_1,\dots,X_n;Y)=\VCat(\und\cv,\und\cv)^\op(L^{X_1}\cdot\ldots\cdot L^{X_n},L^Y)
=\VCat(\und\cv,\und\cv)(L^Y,L^{X_n}\circ\dots\circ L^{X_1}).
\]
Identities and composition coincide with those of the multicategory
associated with the strict monoidal category
\(\VCat(\und\cv,\und\cv)^\op\). Note that by \propref{prop-repr}
there is a bijection
\[
\Gamma:\mcV(X_1,\dots,X_n;Y)\to(V\circ L^{X_n}\circ\dots\circ
L^{X_1})Y,\quad f\mapsto (Vf_Y)1_Y.
\]

\begin{theorem}
  The multicategory \(\mcV\) is closed and has a unit object. The
  underlying closed category of \(\mcV\) is isomorphic to \(\cv\).
\end{theorem}

\begin{proof}
  First, let us check that the multicategory \(\mcV\) is closed. By
  \propref{prop-undCXZ-implies-closedness}, it suffices to prove that for
  each pair of objects \(X\) and \(Z\) there exist an internal
  \(\Hom\)\n-object \(\und\mcV(X;Z)\) and an evaluation morphism
  \(\ev^\mcV_{X;Z}:X,\und\mcV(X;Z)\to Z\) such that the map
  \[
  \varphi:\mcV(Y_1,\dots,Y_n;\und\mcV(X;Z))\to\mcV(X,Y_1,\dots,Y_n;Z),
  \quad f\mapsto(1_X,f)\cdot\ev^\mcV_{X;Z},
  \label{equ-varphi-X-Y1Yn-Z}
  \]
  is bijective, for each sequence of objects \(Y_1\), \dots,
  \(Y_n\). We set \(\und\mcV(X;Z)=\und\cv(X,Z)\). The evaluation map
  \(\ev^\mcV_{X;Z}:X,\und\mcV(X;Z)\to Z\) is by definition a
  \(\cv\)\n-natural transformation \(L^Z\to L^{\und\cv(X,Z)}\circ
  L^X\).  We define it by requesting
  \((V(\ev^\mcV_{X;Z})_Z)1_Z=1_{\und\cv(X,Z)}\) (we extensively use
  the representation theorem for \(\cv\)\n-functors in the form of
  \propref{prop-repr}). Let us check that the map \(\varphi\) is
  bijective. Note that the codomain of \(\varphi\) identifies via the
  map \(\Gamma\) with the set \((V\circ L^{Y_n}\circ\dots\circ
  L^{Y_1}\circ L^X)Z\), and that the domain of \(\varphi\) identifies
  via \(\Gamma\) with the set
  \[
  (V\circ L^{Y_n}\circ\dots\circ L^{Y_1})\und\cv(X,Z)=(V\circ L^{Y_n}\circ\dots\circ L^{Y_1}\circ L^X)Z.
  \]
  The bijectivity of \(\varphi\) follows readily from the diagram
  \[
  \begin{xy}
    (-25,9)*+{\mcV(Y_1,\dots,Y_n;\und\mcV(X;Z))}="1";
    (25,9)*+{\mcV(X,Y_1,\dots,Y_n;Z)}="2";
    (0,-9)*+{(V\circ L^{Y_n}\circ\dots\circ L^{Y_1}\circ L^X)Z}="3";
    {\ar@{->}^-{\varphi} "1";"2"};
    {\ar@{->}_-{\Gamma} "1";"3"};
    {\ar@{->}^-{\Gamma} "2";"3"};
  \end{xy}
  \]
  whose commutativity we are going to establish. Indeed, take an element
  \(f\in\mcV(Y_1,\dots,Y_n;\und\mcV(X;Z))\), i.e., a \(\cv\)\n-natural
  transformation \(f:L^{\und\cv(X,Z)}\to L^{Y_n}\circ\dots\circ
  L^{Y_1}\). Then \(\varphi(f)\) is given by the composite
  \[
  L^Z\rto{\ev^\mcV_{X;Z}}L^{\und\cv(X,Z)}\circ L^X\rto{f
    L^X}L^{Y_n}\circ\dots\circ L^{Y_1}\circ L^X.
  \]
  Therefore, \(\Gamma\varphi(f)\) is equal to
  \[
  \bigl(V((f L^X)\circ\ev^\mcV_{X;Z})_Z\bigr)1_Z=\bigl(V(f
  L^X)_Z\bigr)\bigl(V\ev^\mcV_{X;Z}\bigr)1_Z=(Vf_{\und\cv(X,Z)})1_{\und\cv(X,Z)}=\Gamma(f).
  \]
  Thus we conclude that \(\mcV\) is a closed multicategory. 

  Let us check that \(\unito\in\Ob\cv\) is a unit object of \(\mcV\). By
  definition, a morphism \(u:()\to\unito\) is a \(\cv\)\n-natural
  transformation \(L^\unito\to\Id\). We let it be equal to \(i^{-1}\),
  which is a \(\cv\)\n-natural transformation by
  \cite[Proposition~8.5]{EK}. Then for each object \(X\) of \(\cv\)
  holds
  \[
  \und\mcV(u;1)=(u,1)\cdot\ev^\mcV_{\unito;X}:\und\mcV(\unito;X)\to X,
  \]
  i.e., \(\und\mcV(u;1)\) is the \(\mcV\)\n-natural transformation
  \[
  L^X\rto{\ev^\mcV_{\unito;X}}L^{\und\cv(\unito;X)}\circ
  L^\unito\rto{L^{\und\cv(\unito,X)}u}L^{\und\cv(\unito,X)}.
   \]
  We claim that it coincides with \(L^{i^{-1}_X}\) and hence is
  invertible. Indeed, applying \(\Gamma\) to the above composite we
  obtain
  \[
  \bigl(V((L^{\und\cv(\unito,X)}u)\circ\ev^\mcV_{\unito;X})_X\bigr)1_X=
  \bigl(V(L^{\und\cv(\unito,X)}u)_X\bigr)\bigl(V(\ev^\mcV_{\unito;X})_X\bigr)1_X
  =\cv(\und\cv(\unito,X),u_X)1_{\und\cv(\unito,X)}=u_X=i^{-1}_X.
  \]

  Let us now describe the underlying closed category of the closed
  multicategory \(\mcV\). Its objects are those of \(\cv\), and for each
  pair of objects \(X\) and \(Y\) the set of morphisms from \(X\) to
  \(Y\) is \(\mcV(X;Y)=\VCat(\und\cv,\und\cv)(L^Y,L^X)\). The unit object
  is \(\unito\) and the internal \(\Hom\)\n-object \(\und\mcV(X;Y)\)
  coincides with \(\und\cv(X,Y)\). For each object \(X\), the identity
  morphism \(1^{\und\mcV}_X:()\to\und\mcV(X;X)\), i.e., a
  \(\cv\)\n-natural transformation \(L^{\und\cv(X,X)}\to\Id\), is found
  from the equation
  \[
  \bigl[
  X\rto{1_X,1^{\und\mcV}_X}X,\und\mcV(X;X)\rto{\ev^\mcV_{X;X}}X
  \bigr]=1_X,
  \]
  or equivalently from the equation
  \[
  \bigl[
  L^X\rto{\ev^\mcV_{X;X}}L^{\und\cv(X,X)}\circ L^X\rto{1^{\und\mcV}_X L^X}L^X
  \bigr]=\id.
  \]
  Applying \(\Gamma\) to both sides we find that
  \[
  \bigl(V((1^{\und\mcV}_X L^X)\circ\ev^\mcV_{X;X})_X\bigr)1_X=\bigl(V(1^{\und\mcV}_X)_{\und\cv(X,X)}\bigr)\bigl(V(\ev^\mcV_{X;X})_X\bigr)1_X
  =V(1^{\und\mcV}_X)_{\und\cv(X,X)}1_{\und\cv(X,X)}=1_X.
  \]
  Here
  \(V(1^{\und\mcV}_X)_{\und\cv(X,X)}:\cv(\und\cv(X,X),\und\cv(X,X))\to\cv(X,X)\).
  The morphism \(j_X\) of the underlying closed category of \(\mcV\) is a
  \(\cv\)\n-natural transformation \(L^{\und\cv(X,X)}\to L^\unito\); it
  is found from the equation
  \[
  \bigl[
  L^{\und\cv(X,X)}\rto{j_X}L^\unito\rto{u}\Id
  \bigr]=1^{\und\mcV}_X.
  \]
  Applying \(\Gamma\) to both sides we obtain
  \[
  \bigl(V(u\circ j_X)_{\und\cv(X,X)}\bigr)1_{\und\cv(X,X)}=V(1^{\und\mcV}_X)_{\und\cv(X,X)}1_{\und\cv(X,X)},
  \]
  i.e.,
  \[
  \bigl(Vi^{-1}_{\und\cv(X,X)}\bigr)\bigl(V(j_X)_{\und\cv(X,X)}1_{\und\cv(X,X)}\bigr)=1_X,
  \]
  or equivalently
  \[
  (V(j_X)_{\und\cv(X,X)})1_{\und\cv(X,X)}=(Vi_{\und\cv(X,X)})1_X=j_X,
  \]
  where the last equality is the axiom CC5'. Therefore,
  \(j_X=L^{j_X}:L^{\und\cv(X,X)}\to L^\unito\). It follows from the
  construction that \(i_X\) for the underlying closed category of
  \(\mcV\) is \((\und\mcV(u;1))^{-1}=(L^{i^{-1}_X})^{-1}=L^{i_X}\).

  Let us compute the morphism
  \(L^X_{YZ}:\und\mcV(Y;Z)\to\und\mcV(\und\mcV(X;Y);\und\mcV(X;Z))\).
  Before we do that, note that \(\ev^\mcV_{X;Y}:X,\und\mcV(X;Y)\to Y\) is
  the \(\cv\)\n-natural transformation \(L^Y\to L^{\und\cv(X,Y)}\circ
  L^X\) with components
  \[
  (\ev^\mcV_{X;Y})_Z=L^X_{YZ}:\und\cv(Y,Z)\to\und\cv(\und\cv(X,Y),\und\cv(X,Z)).
  \]
  In other words, \(\ev^\mcV_{X;Y}=L^X_{Y,-}\). Indeed, applying
  \(\Gamma\) to both side of the equation in question we obtain an
  equivalent equation
  \[
  (V(\ev^\mcV_{X;Y})_Y)1_Y=(VL^X_{YY})1_Y.
  \]
  Since \(VL^X=\cv(X,-)\), it follows that
  \((VL^X_{YY})1_Y=1_{\und\cv(X,Y)}\), so that the obtained equation is
  just the definition of \(\ev^\mcV_{X;Y}\).

  The morphism
  \(L^X_{YZ}:\und\mcV(Y;Z)\to\und\mcV(\und\mcV(X;Y);\und\mcV(X;Z))\) is
  found from the equation
  \begin{multline*}
    \bigl[X,\und\mcV(X;Y),\und\mcV(Y;Z)\rto{1,1,L^X_{YZ}}X,\und\mcV(X;Y),\und\mcV(\und\mcV(X;Y);\und\mcV(X;Z))
    \\
    \hfill
    \rto{1,\ev^\mcV_{\und\mcV(X;Y);\und\mcV(X;Z)}}X,\und\mcV(X;Z)\rto{\ev^\mcV_{X;Z}}Z\bigr]
    \quad
    \\
    =\bigl[
    X,\und\mcV(X;Y),\und\mcV(Y;Z)\rto{\ev^\mcV_{X;Y},1}Y,\und\mcV(Y;Z)\rto{\ev^\mcV_{Y;Z}}Z
    \bigr],
  \end{multline*}
  or equivalently
  \begin{multline*}
    \bigl[
    L^Z\rto{\ev^\mcV_{X;Z}}L^{\und\cv(X,Z)}\circ L^X\rto{\ev^\mcV_{\und\cv(X,Y);\und\cv(X,Z)}L^X}
    L^{\und\cv(\und\cv(X,Y),\und\cv(X,Z))}\circ L^{\und\cv(X,Y)}\circ L^X
    \\
    \hfill
    \rto{L^X_{YZ} L^{\und\cv(X,Y)} L^X}L^{\und\cv(Y,Z)}\circ
    L^{\und\cv(X,Y)}\circ L^X
    \bigr]
    \quad
    \\
    =\bigl[
    L^Z\rto{\ev^\mcV_{Y;Z}}L^{\und\cv(Y,Z)}\circ L^Y\rto{L^{\und\cv(Y,Z)}\ev^\mcV_{X;Y}}
    L^{\und\cv(Y,Z)}\circ L^{\und\cv(X,Y)}\circ L^X
    \bigr].
  \end{multline*}
  Applying \(\Gamma\) to both side of the above equation we obtain
  \begin{multline*}
    \bigl(V(L^X_{YZ})_{\und\cv(\und\cv(X,Y),\und\cv(X,Z))}\bigr)
    \bigl(V(\ev^\mcV_{\und\cv(X,Y);\und\cv(X,Z)})_{\und\cv(X,Z)}\bigr)\bigl(V(\ev^\mcV_{X;Z})_Z\bigr)1_Z
    \\
    =\cv(\und\cv(Y,Z),(\ev^\mcV_{X;Y})_Z)(V(\ev^\mcV_{Y;Z})_Z)1_Z,
  \end{multline*}
  or equivalently
  \[
  \bigl(V(L^X_{YZ})_{\und\cv(\und\cv(X,Y),\und\cv(X,Z))}\bigr)1_{\und\cv(\und\cv(X,Y),\und\cv(X,Z))}=
  (\ev^\mcV_{X;Y})_Z=L^X_{YZ}.
  \]
  In other words, \(L^X_{YZ}\) for the underlying closed category of
  \(\mcV\) is \(L^{L^X_{YZ}}\).

  Let us denote the underlying closed category of the multicategory
  \(\mcV\) by the same symbol. There is a closed functor
  \((L,1,1):\cv\to\mcV\), where \(L:\cv\to\mcV\) is given by
  \(X\mapsto X\), \(f\mapsto L^f\), and the morphisms
  \(\und\cv(X,Y)\to\und\mcV(X;Y)\) and \(\unito\to\unito\) are the
  identities. The axioms CF1--CF3 follow readily from the above
  description of the closed category \(\mcV\). Clearly, the functor
  \(L\) is an isomorphism. The theorem is proven.
\end{proof}

\small

\textsc{Department of Mathematics and Statistics, York University,
  4700 Keele Street, To\-ron\-to, ON, Canada, M3J~1P3}

\textit{E-mail address:} \texttt{manzyuk@mathstat.yorku.ca}

\end{document}